\documentclass[oneside,english]{amsart}
\usepackage[T1]{fontenc}
\usepackage[latin9]{inputenc}
\usepackage{geometry}
\usepackage{amstext}
\usepackage{amsthm}
\usepackage{amssymb}

\makeatletter
\numberwithin{equation}{section}
\numberwithin{figure}{section}
\theoremstyle{plain}
\newtheorem{thm}{\protect\theoremname}
\theoremstyle{definition}
\newtheorem{defn}[thm]{\protect\definitionname}
\theoremstyle{remark}
\newtheorem{rem}[thm]{\protect\remarkname}
\theoremstyle{definition}
\newtheorem{example}[thm]{\protect\examplename}
\theoremstyle{plain}
\newtheorem{prop}[thm]{\protect\propositionname}
\theoremstyle{plain}
\newtheorem{lem}[thm]{\protect\lemmaname}
\theoremstyle{plain}
\newtheorem{cor}[thm]{\protect\corollaryname}

\makeatother

\usepackage{babel}
\providecommand{\corollaryname}{Corollary}
\providecommand{\definitionname}{Definition}
\providecommand{\examplename}{Example}
\providecommand{\lemmaname}{Lemma}
\providecommand{\propositionname}{Proposition}
\providecommand{\remarkname}{Remark}
\providecommand{\theoremname}{Theorem}

\begin{document}

\title[Path-wise Integration for Fields]{An extension of the sewing lemma to hyper-cubes and hyperbolic equations
driven by multi parameter Young fields }

\author{Fabian A. Harang \\
University of Oslo }

\address{Fabian A. Harang\\
Department of Mathematics, University of Oslo\\
Email: fabianah@math.uio.no}


\thanks{$Acknowledgment$: I am very are grateful for all the assistance,
comments and advice I have received from Professor Frank N. Proske
during the work on this project. Also, I would like to thank Professor
Fred E. Benth and Professor Samy Tindel, whose comments greatly improved
this manuscript. }
\begin{abstract}
This article extends the celebrated sewing lemma to multi-parameter
fields on hyper cubes. We use this to construct Young integrals for
multi-parameter Hölder fields on general domains $\left[0,1\right]^{k}$
with $k\geq1$. Moreover, we show existence, uniqueness and stability
of some particular types of hyperbolic PDE's driven by space-time
Hölder noise in a Young regime. 
\end{abstract}

\maketitle
\tableofcontents{}

\section{Introduction }

In recent years, analysis of stochastic processes based on their path-wise
properties rather than their probabilistic properties has received
much attention. Especially after T. Lyons published the celebrated
article on rough paths (\cite{TLyons}) in the late 1990's, this field
of research exploded. He proved that if one can show (by probabilistic
methods) that for a $\alpha-$Hölder continuous stochastic process
$X$ with $\alpha\in(\frac{1}{3},\frac{1}{2}]$ there exists a Lèvy
area (iterated integral) corresponding to $X$ satisfying some algebraic
and analytical conditions, one can give path-wise solutions to SDE's
of the form 
\begin{equation}
dY_{t}=f\left(Y_{t}\right)dX_{t},\label{eq:SDE}
\end{equation}
 where $f$ is sufficiently smooth. This is an important extension
of similar results known when the regularity parameter $\alpha$ of
$X$ is above $\frac{1}{2}$, where Young integration theory can be
used to prove that a unique solution to Equation (\ref{eq:SDE}) exists.
The main advantage of the path-wise approach to stochastic calculus
is that it does not rely upon martingale properties of the driving
signal, but rather depend on its Hölder continuity. \\

The current article was motivated by the construction of a path-wise
rough integration method for stochastic fields in the framework of
the theory of rough paths. However, as we shall see, the study of
multi-parameter processes requires much more detail than the 1D theory
of SDE's, and we therefore have chosen to divide this project into
two parts. This is the first part, and here we will focus on extending
the sewing lemma to general Hölder fields on domains of the form $\left[0,1\right]^{k}$
for $k\geq1$. We will apply this extension to show how we can construct
Young integrals for multi-parameter Hölder fields. At last, we prove
existence and uniqueness of some hyperbolic SPDE's with quite general
boundary conditions , driven by Hölder noise in a Young regime. The
second article is an ongoing work to provide a framework similar to
that of rough paths for fields based on the sewing lemma that we provide
in this article. \\

Consider two stochastic fields $X,Y:\left[0,1\right]^{k}\rightarrow\mathbb{R}$,
with some Hölder continuity assumption in each of its variables, in
the sense that for each $i\in\left\{ 1,...,k\right\} $ we have estimates
of the form 
\begin{equation}
\sup_{\left(x_{1},...,x_{k}\right)\in\left[0,1\right]^{k},y_{i}\in\left[0,1\right]}|X\left(x_{1},...,y_{i},...,x_{k}\right)-X\left(x_{1},...,x_{i},...,x_{k}\right)|\lesssim|y_{i}-x_{i}|^{\alpha_{i}},\label{eq:simple regularity assumption}
\end{equation}
for some $\alpha\in\left(0,1\right)^{k}$ and similarly for $Y$ with
respect to some $\beta\in\left(0,1\right)^{k}$. We are interested
in a path-wise construction of integrals of the form 
\begin{equation}
\int_{x_{1}}^{y_{1}}\cdots\int_{x_{k}}^{y_{k}}Y\left(z_{1},...,z_{k}\right)X\left(dz_{1},...,dz_{k}\right),\label{eq:rough field integration}
\end{equation}
 where the differential (at least when $X$ is smooth ) can be interpreted
as 
\[
X\left(dz_{1},...,dz_{k}\right)=\frac{\partial^{k}}{\partial x_{1},...,\partial x_{k}}X\left(z_{1},...,z_{k}\right)dz_{1},...,dz_{k}.
\]
 Integrals of this form have been studied from a random field point
of view see e.g. \cite{WalshCai}, but also other meanings have been
given to such integrals, e.g. see \cite{OleFredAlmut}. In recent
years, integrals on the form of Equation (\ref{eq:rough field integration})
are often referred to as integration with respect to Ambit fields,
where $X$ usually is chosen to be some multi-parameter Volterra field
of the form 
\[
X\left(y_{1},...,y_{k}\right)=\int_{A\left(y_{1},..,y_{k}\right)}G\left(y_{1}-z_{1},...,y_{k}-z_{k}\right)L\left(dz_{1},...,dz_{k}\right),
\]
 where $L$ is a Lèvy basis, $G$ is a possibly singular Volterra
kernel, and $A\left(y_{1},..,y_{k}\right)$ is some subset of $\left[0,1\right]^{k}$
depending on $\left(y_{1},...,y_{k}\right)$, see for example \cite{BN-Benth-Ped-Ver}.
If $L$ is chosen to be a Brownian sheet, and the resulting field
$X$ satisfy Equation (\ref{eq:simple regularity assumption}) for
some $\alpha\in\left(0,1\right)^{k}$, Young type integration with
respect to $X$ has been developed by Tindel and Quer-Sardanyons in
\cite{TindelQuer} and also by N. Towghi in \cite{Tow} for two dimensions.
That is, both \cite{TindelQuer} and \cite{Tow} shows that for $k=2$
, if $X$ and $Y$ satisfy Equation (\ref{eq:simple regularity assumption})
for Hölder coefficients $\left(\alpha_{1},\alpha_{2}\right)$ and
$\left(\beta_{1},\beta_{2}\right)$ respectively with $\alpha_{1}+\beta_{1}>1$
and $\alpha_{2}+\beta_{2}>1$ then integrals of the form 
\[
\int\int_{\left[s,t\right]\times\left[x,y\right]}Y\left(z_{1},z_{2}\right)X\left(dz_{1},dz_{2}\right)
\]
\[
:=\lim_{|\mathcal{P}\left[s,t\right]\times\mathcal{\tilde{P}}\left[x,y\right]|\rightarrow0}\sum_{\left[u,v\right]\in\mathcal{P}}\sum_{\left[\tilde{u},\tilde{v}\right]\in\mathcal{\tilde{P}}}Y\left(u,\tilde{u}\right)\int_{u}^{v}\int_{\tilde{u}}^{\tilde{v}}X\left(d\mathbf{z}\right),
\]
exists, and behaves like a two parameter extension of the Young integral.
The former authors apply this 2-d Young integral to study the stochastic
wave equation driven by a fractional Brownian sheet, with regularity
above $\frac{1}{2}$ in each variable. Our aim is to extend this integration
from $\left[0,1\right]^{2}$ to $\left[0,1\right]^{k}$ by extending
the Sewing lemma familiar from rough paths theory.

In the seminal paper ``Rough Sheets'' by M. Gubinelli and K. Chouk
(\cite{GubChouk}) from 2014, the authors studied such integration
with the motivation to analyze the stochastic wave equation in a rough
path framework. Their methodology is based on the insightful ``algebraic
integration'' framework developed by Gubinelli in previous works
(see e.g. \cite{Gubinelli,GubinelliTindel} for an introduction).
However, as the authors of \cite{GubChouk} point out; \\
- ``Unfortunately the construction of a rough integral in a multi-parameter
setting gives rise to a very complicated algebraic structure which
is not well understood at the moment.'' \\
 The article is therefore restricted only to the construction of the
integral, and do not deal with the stochastic wave equations, or other
forms of differential equations. Their construction of the integral
in 2D is also based on an extension of the sewing lemma, but proven
using the framework of the algebraic integration theory. Our approach
is based on a direct extension of Lemma 4.2 in \cite{FriHai}. 

The crucial part of this procedure is the construction of a $\delta$
operator acting on functions $\varXi:\left[0,1\right]^{k}\times\left[0,1\right]^{k}\rightarrow\mathbb{R}$,
which is used to prove the existence of an abstract integration map
$\mathcal{I}$ constructed as a limit of a Riemann type sum with respect
to the functions $\varXi$ of $\alpha\in\left(0,1\right)^{k}$ joint-Hölder
continuity. That is, for $\mathbf{x},\mathbf{y}\in\left[0,1\right]^{k}$
let $\mathcal{P}[x_{i},y_{i}]$ be a partition of $\left[x_{i},y_{i}\right]$
for $i\in\left\{ 1,..,k\right\} $ and set 
\[
\mathcal{I}\left(\varXi\right)\left(\mathbf{x},\mathbf{y}\right)=\lim_{|\mathcal{P}\left[\mathbf{x},\mathbf{y}\right]|\rightarrow0}\sum_{\left[\mathbf{u,v}\right]\in\mathcal{P}}\varXi\left(\mathbf{u},\mathbf{v}\right)
\]
\[
=\lim_{|\mathcal{P}\left[x_{1},y_{1}\right]|\vee,...,\vee|\mathcal{P}\left[x_{k},y_{k}\right]|\rightarrow0}\sum_{\left[u_{1},v_{1}\right]\in\mathcal{P}\left[x_{1},y_{1}\right]},...,\sum_{\left[u_{k},v_{k}\right]\in\mathcal{P}\left[x_{k},y_{k}\right]}\varXi\left(u_{1},...,u_{k},v_{1},...,v_{k}\right),
\]
where $\varXi$ is assumed to satisfy some regularity conditions and
algebraic relations with respect to the operator $\delta$. The main
result of this sewing Lemma is that we obtain an inequality of the
form
\[
|\mathcal{I}\left(\varXi\right)_{\mathbf{x},\mathbf{y}}-\varXi\left(\mathbf{x},\mathbf{y}\right)|\lesssim\parallel\delta\varXi\parallel_{\beta;\left[0,1\right]^{k}}q_{\alpha}\left(\mathbf{x},\mathbf{y}\right)p_{\beta-\alpha}\left(\mathbf{x},\mathbf{y}\right);\,\,\,\mathbf{x},\mathbf{y}\in\left[0,1\right]^{k},
\]
where $\beta\in\left(1,\infty\right)^{k}$ and $q_{\alpha}\left(\mathbf{x},\mathbf{y}\right):=\prod_{i=1}^{k}|x_{i}-y_{i}|^{\alpha_{i}},$
and $p_{\beta-\alpha}\left(\mathbf{x},\mathbf{y}\right):=\sum_{i=1}^{k}|x_{i}-y_{i}|^{\beta_{i}-\alpha_{i}}$,
and the Hölder type norm will be defined properly in the next section.
This type of inequality tells us that if the action of $\delta$ on
$\varXi$ is sufficiently regular, then $\varXi$ is a sufficiently
good local approximation of the integral $\mathcal{I}\left(\varXi\right)$.
As a consequence, if we introduce more terms to regularize $\varXi$
it will not affect the limit of the Riemann type sum.\\

If $X$ satisfies inequality (\ref{eq:simple regularity assumption})
for $1\leq i\leq k$ and $\alpha\in(\frac{1}{3},\frac{1}{2}]^{k}$,
then we define the integral in (\ref{eq:rough field integration})
by 
\[
\int_{x_{1}}^{y_{1}}\cdots\int_{x_{k}}^{y_{k}}f\left(X\left(z_{1},...,z_{k}\right)\right)X\left(dz_{1},...,dz_{k}\right):=\mathcal{I}\left(\varXi\right)_{\mathbf{x},\mathbf{y}},
\]
where $\varXi\left(\mathbf{x},\mathbf{y}\right)=f\left(X\left(\mathbf{\mathbf{x}}\right)\right)\square_{\mathbf{x},\mathbf{y}}X$
where $\square_{\mathbf{x},\mathbf{y}}X=\int_{x_{1}}^{y_{1}}\cdots\int_{x_{k}}^{y_{k}}X\left(dz_{1},...,dz_{k}\right)$
denotes the rectangular (or generalized) increment of $X$, which
will be described in detail in the next section.

With the construction of a Young integral, we continue to show existence
on uniqueness of differential equations driven by Hölder fields on
hyper cubes, that is; for $X:\left[0,1\right]^{k}\rightarrow\mathbb{R}$
, we show existence and uniqueness of the equation 
\[
Y\left(\mathbf{x}\right)=\xi\left(\mathbf{x}\right)+\int_{0}^{x_{1}}\cdots\int_{0}^{x_{k}}f\left(Y\left(\mathbf{z}\right)\right)X\left(d\mathbf{z}\right),\,\,\,\mathbf{x}\in\left[0,1\right]^{k},
\]
where $\mathbf{x}\mapsto\xi\left(\mathbf{x}\right)$ is assumed to
be some Hölder contiguous field, and captures the behavior of the
above equation on the boundary. Such equations can be viewed as a
stochastic extension of problems connected to hyperbolic PDE's that
are given by the following mixed derivative equation 
\[
\partial_{x,y}u=b\left(x,y,u,\partial_{x}u,\partial_{y}u\right),\,\,\,\left(x,y\right)\in\left[0,1\right]^{2}
\]
 subject to the boundary conditions 
\[
\begin{array}{cc}
u\left(x,0\right)=f(x)\\
u\left(0,y\right)=g\left(y\right)\\
f(0)=g\left(0\right),
\end{array}
\]
for some sufficiently smooth functions $b,f,$ and $g$. A particularly
simple example would be to choose $b\left(x,y,u,\partial_{x}u,\partial_{y}u\right)=b\left(u\right),$
and by integrating both sides, we obtain 
\[
u\left(x,y\right)=-u(0,0)+f(x)+g(y)+\int_{0}^{x}\int_{0}^{y}b\left(u\left(z_{1},z_{2}\right)\right)dz_{1}dz_{2}.
\]
If $b$ is Lipschitz, classical analysis arguments give us existence
and uniqueness of this equation. 

Furthermore, we can introduce randomness in the above equation by
letting $b$ be multiplied by a space time white noise $\partial_{x,y}W$.
This results in the non-linear (multiplicative) hyperbolic SPDE 
\[
\partial_{x,y}u=b\left(u\right)\partial_{x,y}W,\,\,\,\left(x,y\right)\in\left[0,1\right]^{2},
\]
where we assume that it is subject to the same boundary conditions
as earlier. This equation can be written in integral form 
\[
u\left(x,y\right)=-u\left(0,0\right)+f\left(x\right)+g\left(y\right)+\int_{0}^{x}\int_{0}^{y}b\left(u\left(z_{1},z_{2}\right)\right)W\left(dz_{1},dz_{2}\right).
\]
Such equations have received a lot of attention in probability theory
for the last 30-40 years, and is tightly connected to the wave equation,
as described in \cite{TindelQuer}. For example Walsh studied this
type of integral equation in a classical probability setting, using
the probabilistic properties of the Brownian sheet $W$ (see \cite{WalshCai}).
With the framework presented in the current article, we are able to
study such problems without using probability theory, by considering
equations on the form 
\[
\frac{\partial^{k}}{\partial x_{1},...,\partial x_{k}}u=b\left(u\right)\frac{\partial^{k}}{\partial x_{1},...,\partial x_{k}}X,\,\,\,\left(x,y\right)\in\left[0,1\right]^{k},
\]
where $X\in\mathcal{C}^{\alpha}\left(\left[0,1\right]^{k};\mathbb{R}\right)$
for $\alpha\in\left(\frac{1}{2},1\right)^{k}$ and $\frac{\partial^{k}}{\partial x_{1},...,\partial x_{k}}X$
is understood as a distributional derivative, and we study the mild
form of this equation given by 
\[
u\left(\mathbf{x}\right)=\xi\left(\mathbf{x}\right)+\int_{\mathbf{0}}^{\mathbf{x}}b\left(u\left(\mathbf{z}\right)\right)X\left(d\mathbf{z}\right),
\]
where the boundary information is collected in $\xi$. We want to
note that this type of equations can be viewed as an extension of
the equations studied on $\left[0,1\right]^{2}$ in \cite{TindelQuer}
but with more general boundary.\\

An interesting note when we study Young integration with respect to
fields is that both \cite{TindelQuer,Tow} study Equation (\ref{eq:rough field integration})
in a Young type of setting, and give a definition of this integral
in $\left[0,1\right]^{2}$. This 2D young integral has been used by
P. Friz et. al to give a canonical construction of iterated integrals
for Gaussian processes with sufficiently nice co-variance structure
(see \cite{FriVic,FriHai,Friz2016} for a detailed exposition in this
direction). For example, the fact that one can only construct (canonically)
the iterated integral of an fBm $B^{H}$ with Hurst parameter $H>\frac{1}{4}$
is proved in this manner. This bound is based on the need for a type
of complementary Young regularity of the co-variance function of the
fBm, in the sense that if $Q:\left[0,1\right]^{4}\rightarrow\mathbb{R}$
is the co-variance function of a fBm, then $Q$ is of $2H$ regularity,
and we need that $2H+2H>1$ for the construction of $\int QdQ$ in
a Young regime. If we have a proper construction of $\int QdQ$ we
can bound the variance of $\int B^{H}dB^{h}$ (constructed in a Riemann
type way) by the variation norm of $\int QdQ$. With the construction
of the sewing Lemma in the current article, we believe that a similar
study of the regularity of co-variance functions of Gaussian fields
will make us able to construct iterated integrals (and similar polynomial
concepts as we will discuss later) of said fields. In particular if
$X:\left[0,1\right]^{k}\rightarrow\mathbb{R}^{d}$ is a Gaussian field,
then the co-variance will be $Q:\left[0,1\right]^{2k}\rightarrow\mathbb{R}^{d\times d}$
, and if this function has sufficient regularity, we believe one can
construct an integral of the form $\int QdQ$ in a Young type way
over $\left[0,1\right]^{2k}$. However, we have not yet had the time
to explore this yet, but hope to do this in the future. \\
This article is structured as follows: 
\begin{itemize}
\item Section 1 and 2 give a general introduction to the problem at hand
and the framework we use, as well as notation. We propose a generalization
of the $\delta$ operator known from rough path theory, and show how
we can decompose this operator into more simple operators to use in
analysis later. 
\item Section 3 extends the sewing Lemma and describes how we can use this
to construct a Young integral for general fields of the form $X:\left[0,1\right]^{k}\rightarrow\mathbb{R}$
for $k\geq1$ of some Hölder continuity. 
\item Section 4 is devoted to the existence and uniqueness of differential
equations driven by Young fields with Hölder coefficients $\alpha$
in $\left(\frac{1}{2},1\right)^{k}$. Stability results are provided
with respect to the solution. 
\end{itemize}

\section{Notation and framework}

Throughout this article, we will work on subsets of $\left[0,1\right]^{k}$
for some $k\geq1$. We will mostly work on subsets of the form $\left[\mathbf{x},\mathbf{y}\right]=\prod_{i=1}^{k}\left[x_{i},y_{i}\right]$,
where $\mathbf{x}=\left(x_{1},..,x_{k}\right)\in\left[0,1\right]^{k}$
and similarly for $\mathbf{y}$. We will frequently write for two
elements $\mathbf{x},\mathbf{y}\in\left[0,1\right]^{k}$ that $\mathbf{x}<\mathbf{y}$,
where we mean $x_{i}<y_{i}\forall i=1,..,k$. In the same way, we
may use the notation $\mathbf{x}\leq\mathbf{y}$ for the criterion
$x_{i}\leq y_{i}\forall i=1,..,k$. When working with integration
over hyper cube domains we will call a domain degenerate if one or
more of the intervals which constitutes the hyper cube has zero length.
The reason is that if we have a hyper cube $\left[\mathbf{x},\mathbf{y}\right]=\prod_{i=1}^{k}\left[x_{i},y_{i}\right]$
where for some $1\leq i\leq k$, $x_{i}=y_{i}$ we have formally for
some integrable function $g$ that 
\[
\int_{\left[\mathbf{x},\mathbf{y}\right]}dg=\int_{x_{1}}^{y_{1}}\cdots\int_{x_{i}}^{x_{i}}\cdots\int_{x_{k}}^{y_{k}}dg=0.
\]
Thus, we often specify that the interval of interest $\left[\mathbf{x},\mathbf{y}\right]$
must satisfy $\mathbf{x}<\mathbf{y}$. Most computations in this article
are based on a generalized increment on hyper cubes. We therefore
give the following definition of an operator generating increments
of functions in $\left[0,1\right]^{k}$. 
\begin{defn}
$\left(generalized\,\,\,increments\right)$ For $\mathbf{x}=\left(x_{1},..,x_{k}\right)\in\left[0,1\right]^{k}$
and $\mathbf{y}=\left(y_{1},..,y_{k}\right)\in\left[0,1\right]^{k}$
we define the operator $V_{i,\mathbf{x}}$ for each $i\in\left\{ 1,...,k\right\} $
by 
\[
V_{i,\mathbf{x}}\mathbf{y}=\left(y_{1},..,y_{i-1},x_{i},y_{i+1},..,y_{k}\right).
\]
For a function $f:\left[0,1\right]^{k}\rightarrow\mathbb{R}$ we define
the generalized increments of $f$ from $\mathbf{x}$ to $\mathbf{y}$
(i.e over $\left[\mathbf{x},\mathbf{y}\right]\subset\left[0,1\right]^{k}$)
by 
\[
\square_{\mathbf{x}}^{k}f\left(\mathbf{y}\right)=\prod_{i=1}^{k}\left(I-V_{i,\mathbf{x}}\right)f\left(\mathbf{y}\right),
\]
where $I$ denotes the identity operator, and $V_{i,\mathbf{x}}f\left(\mathbf{y}\right)=f\left(V_{i,\mathbf{x}}\mathbf{y}\right)$.
Moreover, if $f$ is a smooth field, we have 
\[
\int_{\mathbf{x}}^{\mathbf{y}}f\left(d\mathbf{z}\right)=\int_{\mathbf{x}}^{\mathbf{y}}\frac{\partial^{k}}{\partial z_{1},...,\partial z_{k}}f\left(\mathbf{z}\right)d\mathbf{z}=\square_{\mathbf{x}}^{k}f\left(\mathbf{y}\right).
\]
\end{defn}

\begin{rem}
\label{rem:product of increment}It will be convenient for later to
use the notation for product of generalized increments 
\[
\square_{\mathbf{x}}^{n}f\left(\mathbf{y}\right)\square_{\mathbf{\tilde{x}}}^{n}f\left(\tilde{\mathbf{y}}\right)=\prod_{k=1}^{n}\left(I-V_{k,\mathbf{x}}\right)\prod_{k=1}^{n}\left(I-V_{k,\tilde{\mathbf{x}}}\right)f\left(\mathbf{y}\right)f\left(\mathbf{\tilde{y}}\right)
\]
\[
=\square_{\left(\mathbf{x},\mathbf{\tilde{x}}\right)}^{2n}f\left(\mathbf{y}\right)f\left(\mathbf{\tilde{y}}\right),
\]
 i.e. we use $\square_{\left(\mathbf{x},\mathbf{\tilde{x}}\right)}^{2n}=\prod_{k=1}^{2n}\left(I-V_{k,\mathbf{\left(\mathbf{y},\mathbf{\tilde{y}}\right)}}\right)$
acting on functions on $\left(\mathbb{R}^{n}\right)\times\left(\mathbb{R}^{n}\right)$
in the above way. 
\end{rem}

Also notice the inductive property of the generalized increment, that
is 
\[
\square_{\mathbf{x}}^{n}f\left(\mathbf{y}\right)=\prod_{k=1}^{n-1}\left(I-V_{k,\mathbf{x}}\right)f\left(\mathbf{y}\right)-V_{n,\mathbf{x}}\prod_{k=1}^{n-1}\left(I-V_{k,\mathbf{x}}\right)f\left(\mathbf{y}\right)
\]
\[
=\square_{\mathbf{x'}}^{n-1}f\left(\mathbf{y}',y_{n}\right)-\square_{\mathbf{x'}}^{n-1}f\left(\mathbf{y}',x_{n}\right),
\]
which will become very useful in the later sections. 

For a more compact notation, write $\square_{\mathbf{x},\mathbf{y}}^{k}f=\square_{\mathbf{x}}^{k}f\left(\mathbf{y}\right)$
as the $k$ increment of $f$ from $\mathbf{x}$ to $\mathbf{y}.$ 
\begin{example}
Consider a function $f\left(\mathbf{x}\right)=\prod_{i=1}^{k}f_{i}\left(x_{i}\right),$
where $f_{i}:\left[0,1\right]\rightarrow\mathbb{R}$ is differentiable.
Then 
\[
\square_{\mathbf{x},\mathbf{y}}^{k}f=\prod_{i=1}^{k}\left(f_{i}\left(y_{i}\right)-f_{i}\left(x_{i}\right)\right).
\]
Also notice that by the mean value Theorem we have 
\[
\square_{\mathbf{x},\mathbf{y}}^{k}f=\frac{\partial^{k}}{\partial x_{1},...,\partial x_{k}}f\left(\mathbf{c}\right)\prod_{i=1}^{k}\left(y_{i}-x_{i}\right)
\]
 for some $\mathbf{c}\in\left[\mathbf{x},\mathbf{y}\right]$, the
last relation holds of course for general $f:\left[0,1\right]^{k}\rightarrow\mathbb{R}$
which is at least one time differentiable in each variable on $\left[0,1\right]^{k}$. 
\end{example}

There are two ways to consider Hölder continuous functions in several
variables. One could either consider the space of functions which
satisfy a Hölder continuity condition on each sub-interval of a hyper
cube, or we could define it as the space of functions whose generalized
increment $\square$ satisfies some multiplicative Hölder type of
regularity. In the current article we will use both concepts, and
will therefore give a definition of a space which captures both. 
\begin{defn}
\label{def:two-parameter functions h=0000F6lder space} Let $\alpha\in\left(0,1\right)^{k}$
be a multi-index. The space of Hölder continuous fields $\mathcal{C}^{\alpha}\left(\left[0,1\right]^{k};\mathbb{R}\right)$
is defined as all $Y:\left[0,1\right]^{k}\rightarrow\mathbb{R}$ which
satisfy 
\[
\parallel Y\parallel_{\alpha,\left[0,1\right]^{k}}:=\parallel Y\parallel_{\alpha,\square,\left[0,1\right]^{k}}+\sum_{i=1}^{k}\sup_{\mathbf{x}\in\left[0,1\right]^{k}}\parallel Y_{\mathbf{x},i}\parallel_{\alpha_{i};\left[0,1\right]}<\infty,
\]
where $Y_{\mathbf{x},i}=Y\left(x_{1},..,x_{i-1},\cdot,x_{i+1},..,x_{k}\right)$
and
\[
\parallel Y_{\mathbf{x},i}\parallel_{\alpha_{i};\left[0,1\right]}=\sup_{u\neq v\in\left[0,1\right]}\frac{|Y_{\mathbf{x},i}\left(v\right)-Y_{\mathbf{x},i}\left(u\right)|}{|v-u|^{\alpha_{i}}}
\]
\[
\parallel Y\parallel_{\alpha,\square,\left[0,1\right]^{k}}=\sup_{\mathbf{x},\mathbf{y}\in\left[0,1\right]^{k}\setminus D}\frac{|\square_{\mathbf{x},\mathbf{y}}^{k}Y|}{q_{\alpha}\left(\mathbf{x},\mathbf{y}\right)},
\]
where $q_{\alpha}\left(\mathbf{x},\mathbf{y}\right)=\prod_{i=1}^{k}|y_{i}-x_{i}|^{\alpha_{i}}$,
and 
\[
D=\left\{ \mathbf{x}\in\left[0,1\right]^{k}|x_{i}\neq x_{j}\,\,\,for\,\,\,i\neq j\in\left\{ 1,..,k\right\} \right\} .
\]
. 

Also, let us denote by 
\[
\parallel Y\parallel_{\alpha,+,\left[0,1\right]^{k}}=\sum_{i=1}^{k}\sup_{\mathbf{x}\in\left[0,1\right]^{k}}\parallel Y_{\mathbf{x},i}\parallel_{\alpha_{i};\left[0,1\right]},
\]
i.e. the Hölder semi norm, but without the multiplicative regularity
part. 
\end{defn}

From the above we can see that 
\[
\parallel Y\parallel_{\infty}:=\sup_{\mathbf{x\in}\left[0,1\right]^{k}}|Y\left(\mathbf{x}\right)|\leq|Y\left(\mathbf{0}\right)|+\parallel Y\parallel_{\alpha,+,\left[0,1\right]^{k}},
\]
which will become useful later. 
\begin{prop}
\label{prop:Completeness of Calpha}For $Y\in\mathcal{C}^{\alpha}\left(\left[0,1\right]^{k};\mathbb{R}\right)$
the mapping 
\[
Y\mapsto|Y\left(\mathbf{0}\right)|+\parallel Y\parallel_{\alpha,\left[0,1\right]^{k}}
\]
 induces a proper norm on $\mathcal{C}^{\alpha}\left(\left[0,1\right]^{k};\mathbb{R}\right)$
and makes it a Banach space. 
\end{prop}

An important property of the above Hölder semi-norm on $\mathcal{C}^{\alpha}$
is it's ability to scale. The next proposition is a generalization
of exercise 4.24 in \cite{FriHai}, and will become very useful in
later sections.
\begin{prop}
\label{prop:scaleability of H=0000F6lder norms} Let $\mathbf{0}<\mathbf{T}\in\left(0,1\right)^{k}$
and assume that for all $\rho\in\left(\mathbf{0},\mathbf{1}-\mathbf{T}\right)$
the function $f$ satisfies $\parallel f\parallel_{\alpha,\left[\mathbf{\rho},\rho+\mathbf{T}\right]}\leq M$
(uniformly in $\mathbf{\rho}$) for a fixed $\alpha\in\left(0,1\right)^{k}$,
then 
\[
\parallel f\parallel_{\alpha,\left[0,1\right]^{k}}\leq kM2\left(1\vee p_{\alpha-1}\left(\mathbf{0},\mathbf{T}\right)\right),
\]
 where $p_{\gamma}\left(\mathbf{x},\mathbf{y}\right)=\sum_{i=1}^{k}|y_{i}-x_{i}|^{\gamma_{i}}$
. 
\end{prop}

\begin{proof}
This is a simple extension of exercise 4.24 in \cite{FriHai}. 
\end{proof}
Earlier we introduced a generalized increment $\square$ ``lifting''
functions from $C\left(\left[0,1\right]^{k}\right)$ to $C\left(\left[0,1\right]^{k}\times\left[0,1\right]^{k}\right)$.
We will also need an operator lifting functions from $C\left(\left[0,1\right]^{k}\times\left[0,1\right]^{k}\right)$
to $C\left(\left[0,1\right]^{k}\times\left[0,1\right]^{k}\times\left[0,1\right]^{k}\right)$.
The idea is familiar to the classical theory of rough paths, and we
will give a closer description of this operator in the next definition.
\begin{defn}
\label{def:delta operator}Consider a function $f:\left[0,1\right]^{k}\times\left[0,1\right]^{k}\rightarrow\mathbb{R}$
and define the operator 
\[
V_{i,\mathbf{v},\mathbf{z}}\left(\mathbf{x},\mathbf{y}\right):=\left(\left(x_{1},..,x_{i-1},v_{i},x_{i+1},..,x_{k}\right),\left(y_{1},..,y_{i-1},z_{i},y_{i+1},..,y_{k}\right)\right)
\]
 and let $V_{i,\mathbf{v},\mathbf{z}}f\left(\mathbf{x},\mathbf{y}\right)=f\left(V_{i,\mathbf{v},\mathbf{z}}\left(\mathbf{x},\mathbf{y}\right)\right)$
(notice that $V$ is the same $V$ we have used before for increments,
but extended to $\left[0,1\right]^{k}\times\left[0,1\right]^{k}$).
The Chen operator (or delta $\left(\delta\right)$ operator) 
\[
\delta^{\theta}:C\left(\left[0,1\right]^{k}\times\left[0,1\right]^{k}\right)\rightarrow C\left(\left[0,1\right]^{k}\times\left[0,1\right]^{k}\times\left[0,1\right]^{k}\right)
\]
 around the points $\mathbf{x},\mathbf{y}\in\left[0,1\right]^{k}$
centered at $\mathbf{z\in}\left[\mathbf{x},\mathbf{y}\right]$ is
defined by 
\[
\delta_{\mathbf{z}}^{\theta}f\left(\mathbf{x},\mathbf{y}\right)=\prod_{i\in\theta}\left(I-\psi_{i}\left(\mathbf{z}\right)\right)f\left(\mathbf{x},\mathbf{y}\right),
\]
where $\psi_{i}\left(\mathbf{z}\right)=V_{i,\mathbf{x},\mathbf{z}}+V_{i,\mathbf{z},\mathbf{y}},$
and $\theta$ is a set of length $1\leq l\leq k$ of elements from
$\left\{ 1,..,k\right\} ,$ i.e. 
\[
\theta=\left\{ \theta_{1},..,\theta_{l}\right\} \in\left\{ 1,..,k\right\} ^{l}.
\]
Also, for some $i\in\left\{ 1,..,k\right\} $, we can choose $\theta=i$
and we get the special case  
\[
\delta_{\mathbf{z}}^{\left(i\right)}f\left(\mathbf{x},\mathbf{y}\right)=\left(I-V_{i,\mathbf{x},\mathbf{z}}-V_{i,\mathbf{z},\mathbf{y}}\right)f\left(\mathbf{x},\mathbf{y}\right)
\]
\[
=f\left(\mathbf{x},\mathbf{y}\right)-f\left(\mathbf{x},\left(y_{1},..,z_{i},..,y_{k}\right)\right)-f\left(\left(x_{1},..,z_{i},..,x_{k}\right),\mathbf{y}\right),
\]
which is the delta operator from Rough Path theory with respect to
the variables $x_{i}<z_{i}<y_{i}$ on some domain $\left[0,1\right].$
\end{defn}

Throughout the article, we will often omit the dependence on $\theta$,
when $\theta=\left\{ 1,..,k\right\} $ which can be considered the
canonical choice of $\delta$ for a hyper-cube $\left[0,1\right]^{k}$.
In other words, we often write $\delta_{\mathbf{z}}:=\delta_{\mathbf{z}}^{\theta}$
when $\theta=\left\{ 1,..,k\right\} $. \\
Next we will present a way to decompose the delta operator into a
sum of $\delta$ operators in only one variable. As this Lemma is
a crucial step in later calculations, we will try to give a very detailed
proof, such that the ideas come out clear, and the proof is transparent. 
\begin{lem}
\label{lem:Deomposition of delta}Let $\delta^{\theta}$ be defined
as above with $\theta=\left\{ 1,..,k\right\} $. Then 
\[
\delta^{\theta}=\sum_{\tilde{\theta}\subset\theta}\left(-1\right)^{|\tilde{\theta}|-1}\sum_{i\in\tilde{\theta}}\prod_{j\in\tilde{\theta}\setminus\left\{ \tilde{\theta}{}_{1},...,\tilde{\theta}_{i}\right\} }\psi_{j}\delta^{\left(i\right)},
\]
where $|\tilde{\theta}|$ is the number of elements in $\tilde{\theta}$,
and where the $\tilde{\theta}$ represents the $2^{k}-1$ possible
different subsets of $\theta$ (excluding $\left\{ \emptyset\right\} $),
and $\tilde{\theta}_{i}$ is the $i'$th component in $\tilde{\theta}.$ 
\end{lem}

\begin{proof}
Let us write $\delta^{\theta}=\prod_{i\in\theta}\left(I-\psi_{i}\right)$,
where $\psi_{i}\left(\mathbf{z}\right)f\left(\mathbf{x},\mathbf{y}\right)=\left(V_{i,\mathbf{x},\mathbf{z}}+V_{i,\mathbf{z,\mathbf{y}}}\right)f\left(\mathbf{x},\mathbf{y}\right)$
for some $f\in C\left(\left[0,1\right]^{2k}\right)$. As the proof
is of a more technical nature, we prove the claim first for $k=2$
to get an intuition about the result. Consider $\delta^{\left(1,2\right)}$
from Definition \ref{def:delta operator}, and see that 
\[
\delta^{\left(1,2\right)}=\delta^{\left(1\right)}\delta^{\left(2\right)}=\left(I-\psi_{1}\right)\left(I-\psi_{2}\right)
\]
\[
=I-\psi_{1}I-\psi_{2}I+\psi_{1}\psi_{2}.
\]
Now add and subtract $I$ to the RHS above, and use that
\[
2\left(I-\psi_{1}\psi_{2}\right)=\delta^{\left(1\right)}+\delta^{\left(2\right)}\psi_{2}+\delta^{\left(1\right)}+\delta^{\left(2\right)}\psi_{1}
\]
to find that 
\[
\delta^{\left(1,2\right)}=\left(I-\psi_{1}\right)+\left(I-\psi_{2}\right)-\left(I-\psi_{1}\psi_{2}\right)
\]
\[
=\delta^{\left(1\right)}+\delta^{\left(2\right)}-\frac{1}{2}\left(\delta^{\left(1\right)}+\delta^{\left(1\right)}\psi_{2}+\delta^{\left(2\right)}+\delta^{\left(2\right)}\psi_{1}\right)
\]
\[
=\frac{1}{2}\left(\delta^{\left(1\right)}+\delta^{\left(2\right)}-\delta^{\left(1\right)}\psi_{2}-\delta^{\left(2\right)}\psi_{1}\right),
\]
thus for $k=2$ we can decompose $\delta^{\left(1,2\right)}$ into
delta operators in each variable, i.e $\delta^{\left(1\right)}$ and
$\delta^{\left(2\right)}$.

Let us also show that it holds for $\delta^{\left(1,2,3\right)}$,
by using another technique. By direct computations of $\delta^{\left(1,2,3\right)}$,
it is easily seen that 
\[
\delta^{\left(1,2,3\right)}=\delta^{\left(3\right)}\delta^{\left(1,2\right)}=\left(I-\psi_{3}\right)\left(I-\psi_{1}I-\psi_{2}I+\psi_{1}\psi_{2}\right)
\]
\[
=I-\psi_{3}-\left(I-\psi_{3}\right)\psi_{1}-\left(I-\psi_{3}\right)\psi_{2}+\left(I-\psi_{3}\right)\psi_{1}\psi_{2}
\]
\[
=I-\psi_{3}-\psi_{2}-\psi_{1}+\psi_{1}\psi_{3}+\psi_{3}\psi_{2}+\psi_{1}\psi_{2}-\psi_{1}\psi_{2}\psi_{3}
\]
\[
=I+\sum_{\tilde{\theta}\subset\left\{ 1,2,3\right\} }\prod_{j\in\tilde{\theta}}\left(-\psi_{j}\right),
\]
 where there are $2^{3}-1=7$ possible subsets $\tilde{\theta}\subset\theta=\left\{ 1,2,3\right\} $
i.e. $\tilde{\theta}=\left\{ 1\right\} ,$$\tilde{\theta}=\left\{ 2\right\} ,...,\tilde{\theta}=\left\{ 1,2,3\right\} $.
Now, there are $3$ subsets which consists of an even number of elements
in them, i.e $\left\{ 1,2,\right\} ,$$\left\{ 2,3\right\} ,\left\{ 1,3\right\} $
and $4$ subsets which consists of an odd number of elements i.e.
$\left\{ 1\right\} ,\left\{ 2\right\} ,\left\{ 3\right\} ,\left\{ 1,2,3\right\} $.
We can therefore divide the above sum into the negative indices and
the positive ones, i.e. write
\[
\sum_{\tilde{\theta}\subset\left\{ 1,2,3\right\} }\prod_{j\in\tilde{\theta}}\left(-\psi_{j}\right)=\sum_{\gamma\subset\left\{ 1,2,3\right\} }\prod_{j\in\gamma}\psi_{j}-\sum_{\gamma^{c}\subset\left\{ 1,2,3\right\} }\prod_{j\in\gamma^{c}}\psi_{j}
\]
where $\gamma$ denotes all subsets of $\left\{ 1,2,3\right\} $ with
an even number of elements and $\gamma^{c}$denotes the subsets of
$\left\{ 1,2,3\right\} $ of odd number of elements. Consider now
$\sum_{\gamma\subset\left\{ 1,2,3\right\} }\prod_{j\in\gamma}\psi_{j}$
and add $\#\gamma=3$ (where $\#\gamma$ means the number of subsets
$\gamma$) identity operators and see that 
\[
\sum_{\gamma\subset\left\{ 1,2,3\right\} }\prod_{j\in\gamma}\psi_{j}=3I-\sum_{\gamma\subset\left\{ 1,2,3\right\} }\left(I-\prod_{j\in\gamma}\psi_{j}\right).
\]
 From the simple identity $I-a\tilde{a}=\left(I-a\right)+\left(I-\tilde{a}\right)a$
we can derive by induction the general identity 
\[
I-\prod_{i\in\gamma}a_{i}=\left(I-a_{\gamma_{1}}\right)+\left(I-\prod_{i\in\gamma\setminus\left\{ \gamma_{1}|\right\} }a_{i}\right)a_{\gamma_{1}}
\]
\[
=\left(I-a_{\gamma_{1|}}\right)+\left(I-a_{\gamma_{2}}\right)a_{\gamma_{1}}+\left(I-\prod_{i\in\gamma\setminus\left\{ \gamma_{1},\gamma_{2}\right\} }a_{i}\right)a_{\gamma_{2}}a_{\gamma_{1}},
\]
\begin{equation}
=\sum_{i\in\gamma}\left(I-a_{i}\right)\prod_{j\in\gamma\setminus\left\{ \gamma{}_{1},...,\gamma_{i}\right\} }a_{j},\label{eq:identity}
\end{equation}
 where $\gamma\setminus\left\{ \gamma{}_{1},...,\gamma_{i}\right\} =\left\{ \gamma_{k}\in\gamma|\gamma_{k}\neq\gamma_{1},...,\gamma_{i}\right\} $.
Using this identity for $\sum_{\gamma\subset\left\{ 1,2,3\right\} }\left(I-\prod_{j\in\gamma}\psi_{j}\right)$,
we can see that 
\[
\sum_{\gamma\subset\left\{ 1,2,3\right\} }\prod_{j\in\gamma}\psi_{j}=3I-\sum_{\gamma\subset\left\{ 1,2,3\right\} }\sum_{i\in\gamma}\left(I-\psi_{i}\right)\prod_{j\in\gamma\setminus\left\{ \gamma{}_{1},...,\gamma_{|\gamma|-i}\right\} }\psi_{j}
\]
\[
=3I-\sum_{\gamma\subset\left\{ 1,2,3\right\} }\sum_{i\in\gamma}\delta^{\left(i\right)}\prod_{j\in\gamma\setminus\left\{ \gamma{}_{1},...,\gamma_{i}\right\} }\psi_{j}.
\]
By the same method using that there are $4$ subsets $\gamma^{c}$,
we can see that the negative part 
\[
\sum_{\gamma^{c}\subset\left\{ 1,2,3\right\} }\prod_{j\in\gamma^{c}}\psi_{j},=4I-\sum_{\gamma^{c}\subset\left\{ 1,2,3\right\} }\sum_{i\in\gamma^{c}}\delta^{\left(i\right)}\prod_{j\in\gamma^{c}\setminus\left\{ \gamma^{c}{}_{1},...,\gamma_{i}^{c}\right\} }\psi_{j}.
\]
Combining the two calculations, we get 
\[
I+\sum_{\gamma^{c}\subset\left\{ 1,2,3\right\} }\sum_{\gamma\subset\left\{ 1,2,3\right\} }\prod_{j\in\gamma}\psi_{j}-\sum_{\gamma^{c}\subset\left\{ 1,2,3\right\} }\prod_{j\in\gamma^{c}}\psi_{j}
\]
\[
=I+3I-\sum_{\gamma\subset\left\{ 1,2,3\right\} }\sum_{i\in\gamma}\delta^{\left(i\right)}\prod_{j\in\gamma\setminus\left\{ \gamma{}_{1},...,\gamma_{i}\right\} }\psi_{j}
\]
\[
-4I+\sum_{\gamma^{c}\subset\left\{ 1,2,3\right\} }\sum_{i\in\gamma^{c}}\delta^{\left(i\right)}\prod_{j\in\gamma^{c}\setminus\left\{ \gamma^{c}{}_{1},...,\gamma_{i}^{c}\right\} }\psi_{j}
\]
\[
=\sum_{\tilde{\theta}\subset\left\{ 1,2,3\right\} }\left(-1\right)^{\left(|\tilde{\theta}|-1\right)}\sum_{i\in\tilde{\theta}}\delta^{\left(i\right)}\prod_{j\in\tilde{\theta}\setminus\left\{ \tilde{\theta}{}_{1},...,\tilde{\theta}_{i}\right\} }\psi_{j},
\]
where $|\tilde{\theta}|$ denotes the number of elements in $\tilde{\theta}$. 

Now we will consider $k\geq1$ with $\theta=\left\{ 1,...,k\right\} $.
By multiplying out the product $\prod_{i\in\theta}\left(I-\psi_{i}\right)$,
we obtain that 

\[
\prod_{i\in\theta}\left(I-\psi_{i}\right)=I+\sum_{\tilde{\theta}}\prod_{i\in\tilde{\theta}}\left(-\psi_{i}\right).
\]
 Note that the sum $\sum_{\tilde{\theta}}\prod_{i\in\tilde{\theta}}\left(-\psi_{i}\right)$
consists of $2^{k}-1$ components of the form $\prod_{i\in\tilde{\theta}}\left(-\psi_{i}\right)$.
Divide this sum into a negative and positive part, i.e 
\[
\sum_{\tilde{\theta}\subset\theta}\prod_{i\in\tilde{\theta}}\left(-\psi_{i}\right)=\sum_{\gamma\subset\theta}\prod_{i\in\gamma}\psi_{i}-\sum_{\gamma^{c}\subset\theta}\prod_{i\in\gamma^{c}}\psi_{i}
\]
 See that $\#\gamma^{c}=2^{k-1}$, and $\#\gamma=2^{k-1}-1$ . Using
the Equation \ref{eq:identity} we find
\[
I-\prod_{i\in\gamma}\psi_{i}=\sum_{i\in\gamma}\delta^{\left(i\right)}\prod_{j\in\gamma\setminus\left\{ \gamma{}_{1},...,\gamma_{i}\right\} }\psi_{j},
\]
we can calculate that 
\[
\sum_{\gamma}\prod_{i\in\gamma}\psi_{i}=\#\gamma I-\sum_{\gamma\subset\theta}\left(I-\prod_{i\in\gamma}\psi_{i}\right)
\]
\[
=\#\gamma I-\sum_{\gamma\subset\theta}\sum_{i\in\gamma}\delta^{\left(i\right)}\prod_{j\in\gamma\setminus\left\{ \gamma{}_{1},...,\gamma_{i}\right\} }\psi_{j}.
\]
 The same can be found for the negative part, 
\[
\sum_{\gamma^{c}\subset\theta}\prod_{i\in\gamma^{c}}\psi_{i}=\#\gamma^{c}I-\sum_{\gamma^{c}\subset\theta}\sum_{i\in\gamma^{c}}\delta^{\left(i\right)}\prod_{j\in\gamma^{c}\setminus\left\{ \gamma^{c}{}_{1},...,\gamma_{i}^{c}\right\} }\psi_{j}
\]

Combining the two, we can see that 
\[
I-\sum_{\tilde{\theta}\subset\theta}\prod_{i\in\tilde{\theta}}\left(-\psi_{i}\right)=\sum_{\tilde{\theta}\subset\theta}\sum_{i\in\tilde{\theta}}\left(-1\right)^{|\tilde{\theta}|-1}\delta^{\left(i\right)}\prod_{j\in\tilde{\theta}\setminus\left\{ \tilde{\theta}{}_{1},...,\tilde{\theta}_{i}\right\} }\psi_{j},
\]
 where the summation over $\tilde{\theta}$ on the RHS is meant over
all possible subsets of $\theta=\left\{ 1,...,k\right\} ,$ and $\left(-1\right)^{|\tilde{\theta}|-1}$
is either $+1$ or $-1$ depending on $\tilde{\theta}.$ At last we
obtain that we can decompose $\delta^{\theta}$ as 
\[
\delta^{\theta}=\sum_{\tilde{\theta}\subset\theta}\sum_{i\in\tilde{\theta}}\left(-1\right)^{|\tilde{\theta}|-1}\delta^{\left(i\right)}\prod_{j\in\tilde{\theta}\setminus\left\{ \tilde{\theta}{}_{1},...,\tilde{\theta}_{i}\right\} }\psi_{j}.
\]
\end{proof}
It is important to note that the product of operators can be written
shortly as $\delta^{\left(i\right)}\psi_{\tilde{\theta}}^{i}:=\delta^{\left(i\right)}\prod_{j\in\tilde{\theta}\setminus\left\{ \tilde{\theta}{}_{1},...,\tilde{\theta}_{i}\right\} }\psi_{j}$.
The operator $\psi_{\tilde{\theta}}^{i}$ is acting on functions $f:\left[0,1\right]^{2k}\rightarrow\mathbb{R}$
and is shuffling around the variable at positions above $i$ , and
is therefore not ``changing'' the regularity of the function. 

\begin{rem}
\label{rem:Gubi/Chouk DIfference in delta}We want to point out that
the above delta operator and the one defined by Gubinelli and Chouk
\cite{GubChouk} are exactly the same when $k=2$, and it can directly
be seen as a generalization of their work. Consider a function $f:\left[0,1\right]^{2}\times\left[0,1\right]^{2}\rightarrow\mathbb{R}$,
and write $\psi_{i}\left(\mathbf{z}\right)=\left(V_{i,\mathbf{\cdot},\mathbf{z}}+V_{i,\mathbf{z},\mathbf{\cdot}}\right)$
for $i=1,2$; the delta operator $\delta$ considered by Gubinelli
and Chouk is then given by 
\[
\delta_{\mathbf{z}}^{\left(1,2\right)}:=\left(I-\psi_{1}\left(\mathbf{z}\right)\right)\left(I-\psi_{2}\left(\mathbf{z}\right)\right)
\]
\[
=\frac{1}{2}\left(\delta^{\left(1\right)}+\delta^{\left(2\right)}-\delta^{\left(1\right)}\psi_{2}\left(\mathbf{z}\right)-\delta^{\left(2\right)}\psi_{1}\left(\mathbf{z}\right)\right).
\]
\end{rem}

We will be interested in knowing the regularity of the delta operator
when acting on functions $\varXi\in\mathcal{C}_{2}^{\alpha}\left(\left[0,1\right]^{k};\mathbb{R}\right)$,
and we can use the decomposition of the $\delta$ operator from Lemma
\ref{lem:Deomposition of delta} to find a criterion for when the
inequality 
\[
|\delta_{\mathbf{z}}\varXi\left(\mathbf{x},\mathbf{y}\right)|\lesssim q_{\alpha}\left(\mathbf{x},\mathbf{y}\right)p_{\beta-\alpha}\left(\mathbf{x},\mathbf{y}\right)
\]
 is satisfied for $p_{\beta}\left(\mathbf{x},\mathbf{y}\right)=\sum_{i=1}^{k}|y_{i}-x_{i}|^{\beta_{i}}$
and $q_{\alpha}\left(\mathbf{x},\mathbf{y}\right)=\prod_{i=1}^{k}|x_{i}-y_{i}|^{\alpha_{i}}.$
\begin{lem}
\label{lem:sufficient condition}Let $\beta\in\left(1,\infty\right)^{k}$
and $\theta=\left\{ 1,..,k\right\} $. If $\varXi:\left[0,1\right]^{k}\times\left[0,1\right]^{k}\rightarrow\mathbb{R}$
and $\varXi$ satisfies 
\[
|\delta_{\mathbf{z}}^{\left(i\right)}\varXi\left(\mathbf{x},\mathbf{y}\right)|\lesssim q_{\left(\alpha_{1},..,\beta_{i},...,\alpha_{k}\right)}\left(\mathbf{x},\mathbf{y}\right)
\]
for all $i\in\theta$, then 
\[
|\delta_{\mathbf{z}}^{\theta}\varXi\left(\mathbf{x},\mathbf{y}\right)|\lesssim q_{\alpha}\left(\mathbf{x},\mathbf{y}\right)p_{\beta-\alpha}\left(\mathbf{x},\mathbf{y}\right).
\]
\end{lem}

\begin{proof}
From Lemma \ref{lem:Deomposition of delta} we know that for $\mathbf{z\in\left[x,y\right]}$
\[
|\delta_{\mathbf{z}}^{\theta}\varXi\left(\mathbf{x},\mathbf{y}\right)|=|\sum_{\tilde{\theta}\subset\theta}\sum_{i\in\tilde{\theta}}\left(-1\right)^{|\tilde{\theta}|-1}\prod_{j\in\tilde{\theta}\setminus\left\{ \tilde{\theta}{}_{1},...,\tilde{\theta}_{i}\right\} }\psi_{j}\delta^{\left(i\right)}\varXi\left(\mathbf{x},\mathbf{y}\right)|
\]
\[
\leq\parallel\delta^{\left(i\right)}\varXi\parallel_{\left(\alpha_{1},..,\beta_{i},...,\alpha_{k}\right),\square}
\]
\[
\times\sum_{\tilde{\theta}\subset\theta}\sum_{i\in\tilde{\theta}}\left(-1\right)^{|\tilde{\theta}|-1}\prod_{j\in\tilde{\theta}\setminus\left\{ \tilde{\theta}{}_{1},...,\tilde{\theta}_{i}\right\} }\psi_{j}\left(\mathbf{z}\right)q_{\left(\alpha_{1},..,\beta_{i},...,\alpha_{k}\right)}\left(\mathbf{x},\mathbf{y}\right),
\]
where 
\[
\parallel\delta^{\left(i\right)}\varXi\parallel_{\left(\alpha_{1},..,\beta_{i},...,\alpha_{k}\right),\square}=\sup_{\mathbf{x}\neq\mathbf{y}\in\left[0,1\right]^{k}}\frac{|\delta_{\mathbf{z}}^{\left(i\right)}\varXi\left(\mathbf{x},\mathbf{y}\right)|}{q_{\left(\alpha_{1},..,\beta_{i},...,\alpha_{k}\right)}\left(\mathbf{x},\mathbf{y}\right)}.
\]
Notice that the operator $\prod_{j\in\tilde{\theta}\setminus\left\{ \tilde{\theta}{}_{1},...,\tilde{\theta}_{|\tilde{\theta}|-i}\right\} }\psi_{j}\left(\mathbf{z}\right)$
is just changing the variables which is lower than $|\tilde{\theta}|-i$,
therefore 
\[
\sum_{\tilde{\theta}\subset\theta}\sum_{i\in\tilde{\theta}}\left(-1\right)^{|\tilde{\theta}|-1}\prod_{j\in\tilde{\theta}\setminus\left\{ \tilde{\theta}{}_{1},...,\tilde{\theta}_{i}\right\} }\psi_{j}\left(\mathbf{z}\right)q_{\left(\alpha_{1},..,\beta_{i},...,\alpha_{k}\right)}\left(\mathbf{x},\mathbf{y}\right)
\]
\[
\leq Cq_{\left(\alpha_{1},..,\beta_{i},...,\alpha_{k}\right)}\left(\mathbf{x},\mathbf{y}\right).
\]
 
\end{proof}
The next Lemma will prove another important property; remember that
$\square$ is the generalized increment operation, and one wants the
Chen operator $\delta$ and the increment $\square$ to act in a similar
manner as they do in classical rough paths theory, i.e. the $\delta$
operator acting on the generalized increment is $0$.
\begin{lem}
\label{lem:Delta on increment}If $X\in C\left(\left[0,1\right]^{k};\mathbb{R}\right)$,
then for $\theta=\left\{ 1,..,k\right\} $ we have 
\[
\delta_{\mathbf{z}}^{\theta}\square_{\mathbf{x},\mathbf{y}}^{k}X=0.
\]
\end{lem}

\begin{proof}
We write $\mathbf{z}'=\left(z_{1},..,z_{k-1}\right)$ when $\mathbf{z\in}\left[0,1\right]^{k}$
to signify the removal of the $i$'th variable. First, we observe
that 
\[
\left(I-V_{i,\mathbf{x},\mathbf{z}}-V_{i,\mathbf{z},\mathbf{y}}\right)\square_{\mathbf{x},\mathbf{y}}^{k}X=\square_{\mathbf{x},\mathbf{y}}^{k}X-\left(\square_{\mathbf{x},V_{i,\mathbf{z}}\mathbf{y}}^{k}X+\square_{V_{i,\mathbf{z}}\mathbf{x},\mathbf{y}}^{k}X\right),
\]
we can easily see that for $\theta=\left\{ 1,...,k\right\} $ we have
\[
\square_{\mathbf{x},V_{i,\mathbf{z}}\mathbf{y}}^{k}X+\square_{V_{i,\mathbf{z}}\mathbf{x},\mathbf{y}}^{k}X=
\]
\[
\prod_{j\in\theta\setminus\left\{ i\right\} }\left(I-V_{j,\mathbf{x}}\right)\left(\left(I-V_{i,\mathbf{x}}\right)X\left(V_{i,\mathbf{z}}\mathbf{y}\right)+\left(I-V_{i,V_{i,\mathbf{z}}\mathbf{x}}\right)X\left(\mathbf{y}\right)\right),
\]
\[
=\square_{\mathbf{x},\mathbf{y}}^{k}X
\]
 and therefore it is clear that 
\[
\left(I-V_{i,\mathbf{x},\mathbf{z}}-V_{i,\mathbf{z},\mathbf{y}}\right)\square_{\mathbf{x},\mathbf{y}}^{k}X=0,
\]
and furthermore 
\[
\delta_{\mathbf{z}}^{\theta}\square_{\mathbf{x},\mathbf{y}}^{k}X=\prod_{i\in\theta}\left(I-V_{i,\mathbf{x},\mathbf{z}}-V_{i,\mathbf{z},\mathbf{y}}\right)\square_{\mathbf{x},\mathbf{y}}^{k}X=0.
\]
\end{proof}
In the next section we will see how we can use the $\delta$-operator
to prove a generalized sewing Lemma, which then can be used to construct
Young integration for fields.

\section{Sewing Lemma and Young integration for fields \label{sec:Sewing-lemma}}

We will begin this section with a motivation of the construction of
a sewing map for fields on $\left[0,1\right]^{k}$. Let us first discuss
Riemann sums of functions on hyper-cube domains. We are interested
in constructing $k-$fold integrals of the form $\int_{x_{1}}^{y_{1}}\cdots\int_{x_{k}}^{y_{k}}$,
and in the case of smooth $f$ and $g$ on $\left[0,1\right]^{k}$
this $k-$fold integral can be written as 
\[
\int_{x_{1}}^{y_{1}}\cdots\int_{x_{k}}^{y_{k}}f\left(z_{1},...,z_{k}\right)\frac{\partial^{k}}{\partial z_{1},...,\partial z_{k}}g\left(z_{1},...,z_{k}\right)dz_{1},...,dz_{k}
\]
\begin{equation}
=\lim_{|\mathcal{P}^{1}|\rightarrow0},...,\lim_{|\mathcal{P}^{k}|\rightarrow0}\sum_{\left[u_{1},v_{1}\right]\in\mathcal{P}^{1}}\cdots\sum_{\left[u_{1},v_{1}\right]\in\mathcal{P}^{k}}f\left(u_{1},...,u_{k}\right)\label{eq:Riemann sum on grids}
\end{equation}
\[
\times\int_{u_{1}}^{v_{1}}\cdots\int_{u_{k}}^{v_{k}}\frac{\partial^{k}}{\partial z_{1},...,\partial z_{k}}g\left(z_{1},...,z_{k}\right)dz_{1},...,dz_{k}.
\]
Each of the partitions $\mathcal{P}^{1},...,\mathcal{P}^{k}$ are
made from a division of the respective intervals $\left[x_{1},y_{1}\right],...,\left[x_{k},y_{k}\right].$
Therefore, we could consider $grid-like$ partitions of $\left[0,1\right]^{k}$
by a cartesian product of the partitions $\mathcal{P}^{1},...,\mathcal{P}^{k}$.
In this way, we can write a grid like partition $\mathcal{P}\left[\mathbf{x},\mathbf{y}\right]$
of $\prod_{i=1}^{k}\left[x_{i},y_{i}\right]$ as the product 
\[
\mathcal{P}=\prod_{i\in\theta}\mathcal{P}^{i}\left[x_{i},y_{i}\right],
\]
where $\theta=\left\{ 1,..,k\right\} $ and where each element $\left[\mathbf{u},\mathbf{v}\right]\in\mathcal{P}\left[\mathbf{x},\mathbf{y}\right]$
is of the form $\prod_{i=1}^{k}\left[u_{i},v_{i}\right]$ for $\left[u_{i},v_{i}\right]\in\mathcal{P}^{i}$
for $i\in\left\{ 1,...,k\right\} $. If we also write $\int_{\mathbf{x}}^{\mathbf{y}}=\int_{x_{1}}^{y_{1}}\cdots\int_{x_{k}}^{y_{k}}$
as the integral over a hyper cube $\left[\mathbf{x},\mathbf{y}\right]\subset\left[0,1\right]^{k}$,
we can write Equation (\ref{eq:Riemann sum on grids}) in a more compact
form 
\[
\int_{\mathbf{x}}^{\mathbf{y}}f\left(\mathbf{z}\right)D^{\gamma}g\left(\mathbf{z}\right)d\mathbf{z}=\lim_{|\mathcal{P}|\rightarrow0}\sum_{\left[\mathbf{u},\mathbf{v}\right]\in\mathcal{P}}f\left(\mathbf{u}\right)\int_{\mathbf{u}}^{\mathbf{v}}D^{\gamma}g\left(\mathbf{z}\right)d\mathbf{z},
\]
where $D^{\gamma}=\frac{\partial^{k}}{\partial z_{1},...,\partial z_{k}}$
is the multi-index notation of partial derivatives with $\gamma=\left(1,..,1\right)$
(i.e. $1$, $k$-times). Moreover, for a hyper cube, $\left[\mathbf{u,v}\right]$
we define by 
\[
|\left[\mathbf{u,v}\right]|:=\sup_{j=1,..,k}|v_{j}-u_{j}|
\]
Furthermore define the size of the mesh of the partition in the following
way 
\[
|\mathcal{P}\left[\mathbf{x},\mathbf{y}\right]|:=\sup_{\left[\mathbf{u},\mathbf{v}\right]\in\mathcal{P}\left[\mathbf{x},\mathbf{y}\right]}\left\{ |\left[\mathbf{\mathbf{u}},\mathbf{v}\right]|\right\} .
\]

\begin{rem}
In \cite{FrizVIC2011} the authors discuss variation norms of $2D$
functions. One of the observations discussed in the paper is that
grid-like partitions as defined above is not an exhaustive definition
of partitions on hyper cubes, and they show that the study of functions
in a grid like setting in some $\rho-$variation norm is equivalent
to study $\rho+\epsilon$ variation on general partitions of hyper
cubes for some small $\epsilon>0$. From this, we are tempted to believe
that the integral we construct in over grid like partitions, may be
translated to integrals over general partitions at the cost of some
small loss of Hölder regularity. However, we have not checked this
but leave this for future investigations. \\
\end{rem}

Before moving to the theorem of the sewing map, we must first give
a definition of an appropriate space of abstract ``Riemann integrands''
satisfying some regularity conditions. With the decomposition of the
$\delta$-operator in mind, we give the following definition. 
\begin{defn}
Let $\alpha\in\left(0,1\right)^{k},$ $\beta\in\left(1,\infty\right)^{k}$
and $\theta=\left\{ 1,..,k\right\} $. The space $\mathcal{C}_{2}^{\alpha,\beta,\theta}\left(\left[0,1\right]^{2k};\mathbb{R}\right)$
is defined by all functions $\varXi:\left[0,1\right]^{2k}\rightarrow\mathbb{R},$
such that 
\[
\parallel\varXi\parallel_{\alpha,\beta}:=\parallel\varXi\parallel_{\alpha}+\parallel\delta^{\theta}\varXi\parallel_{\beta}
\]
\[
:=\sup_{\mathbf{x}\neq\mathbf{y}\in\left[0,1\right]^{k}}\frac{|\varXi\left(\mathbf{x},\mathbf{y}\right)|}{q_{\alpha}\left(\mathbf{x},\mathbf{y}\right)}+\sum_{i\in\theta}\sup_{\left(\mathbf{x},\mathbf{y}\right)\in\triangle\left(\left[0,1\right]^{k}\right);z_{i}\in\left[x_{i},y_{i}\right]}\frac{|\delta_{z_{i}}^{\left(i\right)}\varXi_{i,\mathbf{x},\mathbf{y}}\left(u,v\right)|}{q_{\left(\alpha_{1},..,\beta_{i},..,\alpha_{k}\right)}\left(\mathbf{x},\mathbf{y}\right)}<\infty,
\]
 where $q_{\gamma}\left(\mathbf{x},\mathbf{y}\right):=\prod_{i=1}^{k}|y_{i}-x_{i}|^{\gamma_{i}}$
for $\gamma\in\mathbb{R}_{+}^{k}$. 
\end{defn}

By the original ideas of Young, the proof of the sewing Lemma in rough
path theory is based on the following property: let $\mathcal{P}\left[x,y\right]$
be a partition of $\left[x,y\right]$ and for some point $z\in\left[x,y\right]$
dividing two sets $\left[z_{-},z\right],\left[z,z^{+}\right]\in\mathcal{P}\left[x,y\right]$
define 
\[
\mathcal{P}\left[x,y\right]\setminus\left\{ z\right\} :=\mathcal{P}\left[x,y\right]\setminus\left\{ \left[z^{-},z\right],\left[z,z^{+}\right]\right\} \cup\left\{ \left[z^{-},z^{+}\right]\right\} .
\]
 We can then study the difference between the two Riemann type sums
\begin{equation}
\sum_{\left[u,v\right]\in\mathcal{P}\left[x,y\right]\setminus\left\{ z\right\} }\varXi\left(u,v\right)-\sum_{\left[u,v\right]\in\mathcal{P}\left[x,y\right]}\varXi\left(u,v\right)\label{eq:Rieman sum pluss delta equation}
\end{equation}
\[
=\varXi\left(z_{-},z_{+}\right)-\varXi\left(z_{-},z\right)-\varXi\left(z,z_{+}\right),
\]
where $z_{-}$ is the largest point in $\mathcal{P}\left[x,y\right]$
smaller than $z$ and $z_{+}$ the smallest point larger than $z$.
The next lemma can be viewed as a direct extension of the sewing lemma
proved by Gubinelli and Chouk for 2D fields in \cite{GubChouk}. 
\begin{lem}
$\left(Sewing\,\,\,lemma\right)$ \label{thm:Sewing map}Consider
a hyper cube $\text{\ensuremath{\left[\mathbf{x},\mathbf{y}\right]}}\subset\left[0,1\right]^{k}$
for $\mathbf{x}\neq\mathbf{y}$ and let $\varXi\in\mathcal{C}_{2}^{\alpha,\beta,\theta}\left(\left[0,1\right]^{k};\mathbb{R}\right)$.
Then there exists a unique continuous map $\mathcal{I}:\mathcal{C}_{2}^{\alpha,\beta,\theta}\left(\left[0,1\right]^{k};\mathbb{R}\right)\rightarrow\mathcal{C}_{+}^{\alpha}\left(\left[0,1\right]^{k};\mathbb{R}\right)$
such that 
\[
|\mathcal{I}\left(\varXi\right)_{\mathbf{x},\mathbf{y}}-\varXi\left(\mathbf{x},\mathbf{y}\right)|\leq C_{\beta,k}\parallel\delta^{\theta}\varXi\parallel_{\beta}q_{\alpha}\left(\mathbf{x},\mathbf{y}\right)p_{\beta-\alpha}\left(\mathbf{x},\mathbf{y}\right),
\]
where $p_{\beta}\left(\mathbf{x},\mathbf{y}\right)=\sum_{i=1}^{k}|y_{i}-x_{i}|^{\beta_{i}}.$
\end{lem}

\begin{proof}
We follow the ideas used in \cite{FriHai} Lemma 4.2, but generalize
to fields. As for uniqueness, assume $J$ and $\tilde{J}$ are two
candidates for $\mathcal{I}\left(\varXi\right)$, then both satisfy
the above inequality, which imply that $\left(J-\tilde{J}\right)\left(\mathbf{0}\right)=0$
and $\left(J-\tilde{J}\right)\left(\mathbf{x},\mathbf{y}\right)\lesssim q_{\alpha}\left(\mathbf{x},\mathbf{y}\right)p_{\beta-\alpha}\left(\mathbf{x},\mathbf{y}\right)$
for $\beta\in\left(1,\infty\right)^{k}$, which tells us that $J-\tilde{J}=0$. 

Let $\mathcal{P}^{i}\left[x_{i},y_{i}\right]$ denote a partition
of $\left[x_{i},y_{i}\right]$ containing $r\geq2$ number of sets
$\forall i\in\left\{ 1,...,k\right\} $. From Lemma 4.2 in \cite{FriHai}
we know that there exists a point $z_{i}$ in $\left[x_{i},y_{i}\right]$
with two corresponding sets $\left[z_{i}^{-},z_{i}\right],\left[z_{i},z_{i}^{+}\right]\in\mathcal{P}^{i}\left[x_{i},y_{i}\right]$
and the following is satisfied
\begin{equation}
|z_{i}^{+}-z_{i}^{-}|\leq\frac{2}{r-1}|y_{i}-x_{i}|,\label{eq:partition division in point inequlaity}
\end{equation}
We will continue the proof in two steps, first we show a maximal inequality
for partition integrals, and then we show non-dependence of the chosen
partition we have used in the integral. The proof is done by induction,
i.e we will show that the claim holds for $\theta=\left\{ 1,2\right\} $
and then generalize to $\theta=\left\{ 1,...,k\right\} $. 

$\left(1\right)$~~~~Let us write $I_{\mathcal{P}^{\theta}}=\sum_{\left[\mathbf{u},\mathbf{v}\right]\in\mathcal{P}^{\theta}}$
as the partition integral (Riemann sum) over some partition defined
by $\mathcal{P}^{\theta}\left[\mathbf{x},\mathbf{y}\right]=\prod_{i\in\theta}\mathcal{P}^{i}\left[x_{i},y_{i}\right]$,
i.e for $\varXi\in\mathcal{C}_{2}^{\alpha,\beta,\theta}\left(\left[0,1\right]^{k};\mathbb{R}\right)$,
we write 
\[
I_{\mathcal{P}^{\theta}\left[\mathbf{x},\mathbf{y}\right]}\varXi:=\sum_{\left[\mathbf{u},\mathbf{v}\right]\in\mathcal{P}^{\theta}}\varXi\left(\mathbf{u},\mathbf{v}\right)
\]
\[
=\sum_{\left[u_{1},v_{1}\right]\in\mathcal{P}\left[x_{1},y_{1}\right]},...,\sum_{\left[u_{k},v_{k}\right]\in\mathcal{P}\left[x_{k},y_{k}\right]}\varXi\left(\mathbf{u},\mathbf{v}\right).
\]
Using this notation, we know that if $\theta=\left(1,2,...,k\right)$
we can write 
\[
I_{\mathcal{P}^{\theta}}=I_{\mathcal{P}^{\left(1\right)}},...,I_{\mathcal{P}^{\left(k\right)}}=:\prod_{i\in\theta}I_{\mathcal{P}^{\left(i\right)}},
\]
where for fixed $\mathbf{x}^{\prime},\mathbf{y}^{\prime}\in\left[0,1\right]^{k-1},$
we have 
\[
I_{\mathcal{P}^{\left(i\right)}}:\mathcal{C}_{2}^{\alpha_{i},\beta_{i},\left(i\right)}\left(\left[0,1\right]^{k};\mathbb{R}\right)\rightarrow\mathcal{C}^{\alpha}
\]
\[
I_{\mathcal{P}^{\left(i\right)}\left[x_{i},y_{i}\right]}\varXi\left(\cdot,\mathbf{x}^{\prime},\mathbf{y}^{\prime}\right)=\sum_{\left[u_{i},v_{i}\right]\in\mathcal{P}^{\left(i\right)}\left[x_{i},y_{i}\right]}\varXi\left(u_{i},v_{i};\mathbf{x}^{\prime},\mathbf{y}^{\prime}\right).
\]
We will mostly write $I_{\mathcal{P}^{\left(i\right)}}=I_{\mathcal{P}^{\left(i\right)}\left[x_{i},y_{i}\right];\left[\mathbf{x},\mathbf{y}\right]}$
when the set we integrate over is else clear. Moreover, denote by
\[
\mathcal{P}^{\left(i\right)}\left[x_{i},y_{i}\right]\setminus\left\{ z_{i}\right\} :=\mathcal{P}^{\left(i\right)}\left[x_{i},y_{i}\right]\setminus\left\{ \left[z_{i}^{-},z_{i}\right],\left[z_{i},z_{i}^{+}\right]\right\} \cup\left\{ \left[z_{i}^{-},z_{i}^{+}\right]\right\} ,
\]
and then define $\mathcal{P}^{\theta}\setminus\left\{ \mathbf{z}\right\} =\prod_{i\in\theta}\mathcal{P}^{\left(i\right)}\left[x_{i},y_{i}\right]\setminus\left\{ z_{i}\right\} $
as the partition $\mathcal{P}^{\theta}$ when we remove a point $\mathbf{z\in\left[\mathbf{x},\mathbf{y}\right]}$.
From Equation (\ref{eq:Rieman sum pluss delta equation}) we can see
that 
\[
I_{\mathcal{P}^{\left(i\right)}\left[x_{i},y_{i}\right]\setminus\left\{ z_{i}\right\} }=I_{\mathcal{P}^{\left(i\right)}\left[x_{i},y_{i}\right]}+\delta_{z_{i}}^{\left(i\right)},
\]
and we can also write 
\begin{equation}
I_{\mathcal{P}^{\theta}\setminus\left\{ \mathbf{z}\right\} }=\prod_{i\in\theta}\left(I_{\mathcal{P}^{\left(i\right)}\left[x_{i},y_{i}\right]}+\delta_{z_{i}}^{\left(i\right)}\right)\label{eq:expansion when removing point}
\end{equation}
 Let us first consider $\theta=\left\{ 1,2\right\} $ then it follows
from rough path theory (see Lemma 4.2 in \cite{FriHai}) and since
$\varXi\in\mathcal{C}_{2}^{\alpha,\beta,\left\{ 1,2\right\} }\left(\left[0,1\right]^{4};\mathbb{R}\right)$,
that 
\[
\mathcal{I}^{\left(i\right)}\left(\varXi\right)\left(\mathbf{x},\mathbf{y}\right)=\lim_{|\mathcal{P}^{\left(i\right)}|\rightarrow0}\sum_{\left[u_{i},v_{i}\right]\in\mathcal{P}^{\left(i\right)}\left[x_{i},y_{i}\right]}\varXi\left(u_{i},v_{i};\mathbf{x}^{\prime},\mathbf{y}^{\prime}\right)
\]
 is well defined for $i\in\theta$. Just as shown above, we can write
\[
\left(I_{\mathcal{P}^{1}\left[x_{1},y_{1}\right]\setminus\left\{ z_{1}\right\} }-I_{\mathcal{P}^{1}}\right)\left(I_{\mathcal{P}^{2}\left[x_{2},y_{2}\right]\setminus\left\{ z_{2}\right\} }-I_{\mathcal{P}^{2}}\right)\varXi
\]
\[
=\left(I_{\mathcal{P}^{\theta}\setminus\left\{ \mathbf{z}\right\} }-I_{\mathcal{P}^{1}}I_{\mathcal{P}^{2}\setminus\left\{ z_{2}\right\} }-I_{\mathcal{P}^{2}}I_{\mathcal{P}^{1}\setminus\left\{ z_{1}\right\} }+I_{\mathcal{P}^{\theta}}\right)\varXi=\delta_{\mathbf{z}}^{\left(1,2\right)}\varXi\left(\mathbf{z}^{-},\mathbf{z}^{+}\right),
\]
and we can see that by Lemma \ref{lem:Deomposition of delta} we have
the decomposition of $\delta^{\left(1,2\right)}$ such that we get
\[
|\left(I_{\mathcal{P}^{\theta}\setminus\left\{ \mathbf{z}\right\} }-I_{\mathcal{P}^{1}}I_{\mathcal{P}^{2}\setminus\left\{ z_{2}\right\} }-I_{\mathcal{P}^{2}}I_{\mathcal{P}^{1}\setminus\left\{ z_{1}\right\} }+I_{\mathcal{P}^{\theta}}\right)\varXi|
\]
\[
\leq\left(\parallel\delta^{\left(1\right)}\varXi\parallel_{\beta_{1}}+\parallel\delta^{\left(2\right)}\varXi\parallel_{\beta_{2}}\right)\left(q_{\beta_{1},\alpha_{2}}\left(\mathbf{z}^{-},\mathbf{z}^{+}\right)+q_{\alpha_{1},\beta_{2}}\left(\mathbf{z}^{-},\mathbf{z}^{+}\right)\right)
\]
\[
=\left(\parallel\delta^{\left(1\right)}\varXi\parallel_{\beta_{1}}+\parallel\delta^{\left(2\right)}\varXi\parallel_{\beta_{2}}\right)q_{\alpha}\left(\mathbf{z}^{-},\mathbf{z}^{+}\right)p_{\beta-\alpha}\left(\mathbf{z}^{-},\mathbf{z}^{+}\right)
\]
\[
\leq\left(\parallel\delta^{\left(1\right)}\varXi\parallel_{\beta_{1}}+\parallel\delta^{\left(2\right)}\varXi\parallel_{\beta_{2}}\right)q_{\alpha}\left(\mathbf{x},\mathbf{y}\right)p_{\beta-\alpha}\left(\mathbf{x},\mathbf{y}\right)
\]
\[
\times\sum_{j=1}^{2}\prod_{i=1}^{2}\left(\frac{2}{r-1}\right)^{\alpha_{i}}\left(\frac{2}{r-1}\right)^{\beta_{j}-\alpha_{j}},
\]
where we also used inequality (\ref{eq:partition division in point inequlaity}).
Furthermore, it is not difficult to see that 
\[
\sum_{j=1}^{2}\prod_{i=1}^{2}\left(\frac{2}{r-1}\right)^{\alpha_{i}}\left(\frac{2}{r-1}\right)^{\beta_{j}-\alpha_{j}}\leq\sum_{j=1}^{2}\left(\frac{2}{r-1}\right)^{\beta_{j}}.
\]
 By iteration we can find that 
\[
\left(Id-I_{\mathcal{P}^{1}}\right)\left(Id-I_{\mathcal{P}^{2}}\right)\varXi\left(\mathbf{x},\mathbf{y}\right)
\]
\[
=\left(\sum_{i=1}^{r}I_{\mathcal{P}^{1}\setminus\left\{ z_{1}^{1},..,z_{1}^{i}\right\} }-I_{\mathcal{P}^{1}\setminus\left\{ z_{1}^{1},..,z_{1}^{i-1}\right\} }\right)\left(\sum_{i=1}^{r}I_{\mathcal{P}^{2}\setminus\left\{ z_{2}^{1},..,z_{2}^{i}\right\} }-I_{\mathcal{P}^{2}\setminus\left\{ z_{2}^{1},..,z_{2}^{i-1}\right\} }\right)\varXi
\]
\[
=\sum_{i,j=1}^{r}\left(I_{\mathcal{P}^{1}\setminus\left\{ z_{1}^{1},..,z_{1}^{i}\right\} }-I_{\mathcal{P}^{1}\setminus\left\{ z_{1}^{1},..,z_{1}^{i-1}\right\} }\right)\left(I_{\mathcal{P}^{2}\setminus\left\{ z_{2}^{1},..,z_{2}^{j}\right\} }-I_{\mathcal{P}^{2}\setminus\left\{ z_{2}^{1},..,z_{2}^{j-1}\right\} }\right)\varXi
\]
\[
=\sum_{i,j=1}^{r}\delta_{z_{1}^{i},z_{2}^{j}}^{\left(1,2\right)}\varXi\left(z_{1}^{i,-},z_{2}^{j,-},z_{1}^{i,+},z_{2}^{j,+}\right)
\]
 By the assumption that $\varXi\in\mathcal{C}_{2}^{\alpha,\beta,\left(1,2\right)}\left(\left[0,1\right]^{4};\mathbb{R}\right)$,
and using inequality (\ref{eq:partition division in point inequlaity}),
we can see that 
\[
|\left(Id-I_{\mathcal{P}^{1}}\right)\left(Id-I_{\mathcal{P}^{2}}\right)\varXi\left(\mathbf{x},\mathbf{y}\right)|.
\]
\[
\lesssim\left(\parallel\delta^{\left(1\right)}\varXi\parallel_{\beta_{1}}+\parallel\delta^{\left(2\right)}\varXi\parallel_{\beta_{2}}\right)q_{\alpha}\left(\mathbf{x},\mathbf{y}\right)p_{\beta-\alpha}\left(\mathbf{x},\mathbf{y}\right)\sum_{j,i=1}^{r}\left(\left(\frac{2}{i}\right)^{\beta_{1}}+\left(\frac{2}{j}\right)^{\beta_{2}}\right)
\]
 
\[
\leq\left(\parallel\delta^{\left(1\right)}\varXi\parallel_{\beta_{1}}+\parallel\delta^{\left(2\right)}\varXi\parallel_{\beta_{2}}\right)q_{\alpha}\left(\mathbf{x},\mathbf{y}\right)p_{\beta-\alpha}\left(\mathbf{x},\mathbf{y}\right)\zeta\left(\beta_{1}\right)\zeta\left(\beta_{2}\right),
\]
where $\zeta$ is the Riemann-Zeta function, and by $\parallel\delta^{\left(i\right)}\varXi\parallel_{\beta_{i}}$
we mean $\parallel\delta^{\left(i\right)}\varXi\parallel_{\left(\beta_{i},\alpha_{j}\right),\square}$
for $i\neq j\in\left\{ 1,2\right\} $. Moreover since $\varXi\in\mathcal{C}_{2}^{\alpha,\beta,\left(1,2\right)}\left(\left[0,1\right]^{4};\mathbb{R}\right),$
we know that 
\[
|I_{\mathcal{P}^{i}}\varXi\left(\mathbf{x},\mathbf{y}\right)-\varXi\left(\mathbf{x},\mathbf{y}\right)|\lesssim\parallel\delta^{\left(i\right)}\varXi\parallel_{\beta_{i}}q_{\alpha}\left(\mathbf{x},\mathbf{y}\right)|x_{i}-y_{i}|^{\beta_{i}-\alpha_{i}}
\]
 for $i\in\theta$. Also, note the relation 
\[
I_{\mathcal{P}^{1,2}}-Id=\left(Id-I_{\mathcal{P}^{1}}\right)\left(Id-I_{\mathcal{P}^{2}}\right)
\]
\[
+\left(I_{\mathcal{P}^{1}}-Id\right)+\left(I_{\mathcal{P}^{2}}-Id\right).
\]
 This lets us conclude that 
\[
|I_{\mathcal{P}^{1,2}\left[\mathbf{x},\mathbf{y}\right]}\varXi\left(\mathbf{x},\mathbf{y}\right)-\varXi\left(\mathbf{x},\mathbf{y}\right)|
\]
\[
\lesssim\left(\parallel\delta^{\left(1\right)}\varXi\parallel_{\beta_{1}}+\parallel\delta^{\left(2\right)}\varXi\parallel_{\beta_{2}}\right)q_{\alpha}\left(\mathbf{x},\mathbf{y}\right)p_{\beta-\alpha}\left(\mathbf{x},\mathbf{y}\right),
\]
 where we have used Next, we will generalize this argument for $k\geq2$,
and $\theta=\left\{ 1,...,k\right\} $. Assume that all integrals
up to order $k-1$ exists, and then do the argument by induction.
More specifically, we assume that for all subsets $\gamma\subset\theta$
the integrals 
\[
\mathcal{I}^{\gamma}\left(\varXi\right)\left(\mathbf{x},\mathbf{y}\right)=\lim_{|\mathcal{P}^{\gamma}|\rightarrow0}\sum_{\left[\mathbf{u},\mathbf{v}\right]\in\mathcal{P}^{\gamma}}\varXi\left(\mathbf{u},\mathbf{v};\mathbf{x},\mathbf{y}\right),
\]
where $\varXi\left(\mathbf{u},\mathbf{v};\mathbf{x},\mathbf{y}\right)$
denotes the Riemann integrand in the variables $\mathbf{u},\mathbf{v}\in\left[0,1\right]^{\gamma}$
and the are also evaluated in the remaining variables of $\mathbf{x}$
and $\mathbf{y}$. Similarly as for the case when $k=2$, we can use
Equation (\ref{eq:expansion when removing point}) , we can iterate
again, and find 
\[
|\prod_{i\in\theta}\left(Id-I_{\mathcal{P}^{\left(i\right)}}\right)\varXi|=|\prod_{i\in\theta}\left(\sum_{j=1}^{r}I_{\mathcal{P}^{\left(i\right)}\setminus\left\{ z_{i}^{1},..,z_{i}^{j}\right\} }-I_{\mathcal{P}^{\left(i\right)}\setminus\left\{ z_{i}^{1},..,z_{i}^{j-1}\right\} }\right)|
\]

\[
=|\sum_{j_{1},..,j_{|\theta|}=1}^{r}\prod_{i\in\theta}\left(I_{\mathcal{P}^{\left(i\right)}\setminus\left\{ z_{i}^{1},..,z_{i}^{j_{i}}\right\} }-I_{\mathcal{P}^{\left(i\right)}\setminus\left\{ z_{i}^{1},..,z_{i}^{j_{i}-1}\right\} }\right)\varXi|
\]
\[
=|\sum_{j_{1},..,j_{|\theta|}=1}^{r}\left(\prod_{i\in\theta}\delta_{z_{i}^{j_{i}}}^{\left(i\right)}\right)\varXi|
\]
\[
\lesssim\parallel\delta^{\theta}\varXi\parallel_{\beta}\sum_{j_{1},..,j_{|\theta|}=1}^{r}q_{\alpha}\left(z_{1}^{j_{1},-},..z_{k}^{j_{|\theta|},-};z_{1}^{j_{1},+},..z_{k}^{j_{|\theta|},+}\right)
\]
\[
\times p_{\beta-\alpha}\left(z_{1}^{j_{1},-},..z_{k}^{j_{|\theta|},-};z_{1}^{j_{1},+},..z_{k}^{j_{|\theta|},+}\right),
\]
where $|\theta|$ is the number of elements in $\theta.$ Using the
bounds from Inequality (\ref{eq:partition division in point inequlaity}),
we can see that 
\[
|\prod_{i\in\theta}\left(Id-I_{\mathcal{P}^{\left(i\right)}}\right)\varXi|\lesssim\left(\sum_{i\in\theta}\parallel\delta^{\left(i\right)}\varXi\parallel_{\beta_{i,}}\right)q_{\alpha}\left(\mathbf{x},\mathbf{y}\right)p_{\beta-\alpha}\left(\mathbf{x},\mathbf{y}\right)\prod_{i\in\theta}\zeta\left(\beta_{i}\right),
\]
where $\zeta$ is the Riemann-Zeta function which is converging since
$\beta_{i}>1$ for all $i=1,..,k$. Furthermore, for all multi-indices
$\gamma\subset\theta=\left\{ 1,2,...,k\right\} $ such that $\gamma$
is a subset of $\theta$ we know by the induction hypothesis that
\[
|\left(I_{\mathcal{P}^{\gamma}}-Id\right)\varXi\left(\mathbf{x},\mathbf{y}\right)|\lesssim\left(\sum_{i\in\gamma}\parallel\delta^{\left(i\right)}\varXi\parallel_{\beta_{i}}\right)q_{\alpha}\left(\mathbf{x},\mathbf{y}\right)p_{\beta-\alpha}^{\gamma}\left(\mathbf{x},\mathbf{y}\right),
\]
 where $p_{\beta}^{\gamma}\left(\mathbf{x},\mathbf{y}\right)=\sum_{i\in\gamma}|x_{i}-y_{i}|^{\beta_{i}}$
, and our goal is to prove that also 
\[
|\left(I_{\mathcal{P}^{\theta}}-Id\right)\varXi\left(\mathbf{x},\mathbf{y}\right)|\lesssim\left(\sum_{i\in\theta}\parallel\delta^{\left(i\right)}\varXi\parallel_{\beta_{i}}\right)q_{\alpha}\left(\mathbf{x},\mathbf{y}\right)p_{\beta-\alpha}\left(\mathbf{x},\mathbf{y}\right).
\]
First, notice that for $\theta=\left\{ 1,...,k\right\} ,$ we can
write 
\begin{equation}
\prod_{i\in\theta}\left(I_{\mathcal{P}^{\left(i\right)}}-Id\right)\varXi=I_{\mathcal{P}^{\theta}}+\left(-1\right)^{|\theta|}Id+\sum_{\gamma\subset\theta}\left(-1\right)^{|\gamma|}I_{\mathcal{P}^{\gamma}}\label{eq:product of delta operator vs sum of lower delta}
\end{equation}
where $|\theta|$ denotes the number of elements in $\theta.$ We
can see that on the RHS the sum $\sum_{\gamma\subset\theta}\left(-1\right)^{|\gamma|}I_{\mathcal{P}^{\gamma}}$
consists of $2^{k}-2$elements, i.e. there exists $2^{k}-2$ ordered
subsets of $\left\{ 1,2,..,k\right\} $ when we exclude $\left\{ 1,2,...,k\right\} $
and $\left\{ \emptyset\right\} $. We will distinguish between two
scenarios, when $|\theta|$ is even and when it is odd. If $|\theta|$
is odd, then the sum $\sum_{\gamma\subset\theta}\left(-1\right)^{|\gamma|}I_{\mathcal{P}^{\gamma}}$
consists of $\left(2^{k}-2\right)/2$ operators with negative signs,
and $\left(2^{k}-2\right)/2$ with positive sign, and we can add and
subtract $\left(2^{k}-2\right)/2$ $Id$-operators, to see that 
\[
\sum_{\gamma\subset\theta}\left(-1\right)^{|\gamma|}I_{\mathcal{P}^{\gamma}}\pm\frac{\left(2^{k}-2\right)}{2}Id=\sum_{\gamma\subset\theta}\left(-1\right)^{|\gamma|+1}\left(I_{\mathcal{P}^{\gamma}}-Id\right).
\]
We can rearrange the above relation, and together with Equation (\ref{eq:product of delta operator vs sum of lower delta})
we obtain that 
\[
|\left(I_{\mathcal{P}^{\theta}}-Id\right)\varXi|\lesssim\sum_{\gamma\subset\theta}|\left(I_{\mathcal{P}^{\gamma}}-Id\right)\varXi|+|\prod_{i\in\theta}\left(I_{\mathcal{P}^{\left(i\right)}}-Id\right)\varXi|.
\]
Combining with our previous estimates, we can see that 
\[
|\left(I_{\mathcal{P}^{\theta}}-Id\right)\varXi\left(\mathbf{x},\mathbf{y}\right)|
\]
\[
\lesssim\sum_{\gamma\subset\theta}\left(\left(\sum_{i\in\gamma}\parallel\delta^{\left(i\right)}\varXi\parallel_{\beta_{i}}\right)p_{\beta-\alpha}^{\gamma}\left(\mathbf{x},\mathbf{y}\right)+\left(\sum_{i\in\theta}\parallel\delta^{\left(i\right)}\varXi\parallel_{\beta_{i}}\right)p_{\beta-\alpha}\left(\mathbf{x},\mathbf{y}\right)\right)q_{\alpha}\left(\mathbf{x},\mathbf{y}\right)
\]
\[
\lesssim2^{k}\left(\sum_{i\in\theta}\parallel\delta^{\left(i\right)}\varXi\parallel_{\beta_{i}}\right)p_{\beta-\alpha}\left(\mathbf{x},\mathbf{y}\right)q_{\alpha}\left(\mathbf{x},\mathbf{y}\right).
\]
On the other hand, if $|\theta|$ is even, then the sum $\sum_{\gamma\subset\theta}\left(-1\right)^{|\gamma|}I_{\mathcal{P}^{\gamma}}$
consists of two more negative operators. Rewriting Equation (\ref{eq:product of delta operator vs sum of lower delta})
we have the equality 
\[
I_{\mathcal{P}^{\theta}}=\prod_{i\in\theta}\left(I_{\mathcal{P}^{\left(i\right)}}-Id\right)\varXi-\left(Id+\sum_{\gamma\subset\theta}\left(-1\right)^{|\gamma|}I_{\mathcal{P}^{\gamma}}\right),
\]
Now, add and subtract $2^{k-1}$ $Id$ operators, and we obtain 
\[
I_{\mathcal{P}^{\theta}}=\prod_{i\in\theta}\left(I_{\mathcal{P}^{\left(i\right)}}-Id\right)\varXi-Id-\sum_{\gamma\subset\theta}\left(-1\right)^{|\gamma|}\left(I_{\mathcal{P}^{\gamma}}-Id\right)+2Id,
\]
 and rewriting this expression and applying it to the integrand $\varXi,$
we conclude just as in the case of odd $|\theta|$ that 
\[
|\left(I_{\mathcal{P}^{\theta}}-Id\right)\varXi|\lesssim2^{k}\left(\sum_{i\in\theta}\parallel\delta^{\left(i\right)}\varXi\parallel_{\beta_{i}}\right)p_{\beta-\alpha}\left(\mathbf{x},\mathbf{y}\right)q_{\alpha}\left(\mathbf{x},\mathbf{y}\right).
\]
This lets us conclude that the above inequality holds for any $\theta=\left\{ 1,...,k\right\} $. 

Also note that the above inequality is independent of which points
that divides up the partition $\mathcal{P}\left[\mathbf{x},\mathbf{y}\right]$.
Therefore, let $\tilde{\mathcal{P}}\left[\mathbf{x},\mathbf{y}\right]$
denote any grid-like partition, and we obtain a maximal inequality,
\[
|\varXi\left(\mathbf{x},\mathbf{y}\right)-I{}_{\mathcal{\tilde{P}}\left[\mathbf{x},\mathbf{y}\right]}\varXi|\lesssim2^{k}\left(\sum_{i\in\theta}\parallel\delta^{\left(i\right)}\varXi\parallel_{\beta_{i}}\right)p_{\beta-\alpha}\left(\mathbf{x},\mathbf{y}\right)q_{\alpha}\left(\mathbf{x},\mathbf{y}\right).
\]
\\
$\left(2\right)\,\,\,\,\,$Next, we will show that we have convergence
of the integral regardless of the chosen grid-like partition. That
is, we will show that the limit $\lim_{|\mathcal{P}\left[\mathbf{x},\mathbf{y}\right]|\rightarrow0}I_{\mathcal{P}\left[\mathbf{x},\mathbf{y}\right]}\varXi$
is independent of the partition $\mathcal{P}\left[\mathbf{x},\mathbf{y}\right]$.
This is equivalent to showing 
\[
\sup_{|\mathcal{P}\left[\mathbf{x},\mathbf{y}\right]|\vee|\mathcal{P}'\left[\mathbf{x},\mathbf{y}\right]|\leq\epsilon}|I_{\mathcal{P}'\left[\mathbf{x},\mathbf{y}\right]}\varXi-I_{\mathcal{P}\left[\mathbf{x},\mathbf{y}\right]}\varXi|\rightarrow0\,\,\,as\,\,\,\epsilon\rightarrow0.
\]
 Assume without loss of generality that $|\mathcal{P}^{\prime}\left[\mathbf{x},\mathbf{y}\right]|\vee|\mathcal{P}\left[\mathbf{x},\mathbf{y}\right]|=|\mathcal{P}\left[\mathbf{x},\mathbf{y}\right]|,$
i.e. $\mathcal{P}^{\prime}\left[\mathbf{x},\mathbf{y}\right]$ refines
$\mathcal{P}\left[\mathbf{x},\mathbf{y}\right]$ in the sense that
for all $i\in\left\{ 1,..,k\right\} $ we have $|\mathcal{P}^{\prime,\left(i\right)}\left[x_{i},y_{i}\right]|\vee|\mathcal{P}^{\left(i\right)}\left[x_{i},y_{i}\right]|=|\mathcal{P}^{\left(i\right)}\left[x_{i},y_{i}\right]|$.
Using that $\mathcal{P}'\left[\mathbf{x},\mathbf{y}\right]=\prod_{i\in\theta}\bigcup_{\left[u_{i},v_{i}\right]\in\mathcal{P}^{\left(i\right)}\left[x_{i},y_{i}\right]}\mathcal{P}^{\prime,\left(i\right)}\left[x_{i},y_{i}\right]\cap\left[u_{i},v_{i}\right]$,
we can use a similar trick as we did before and write 
\[
I_{\mathcal{P}^{\prime}}=\prod_{i\in\theta}I_{\bigcup_{\left[u_{i},v_{i}\right]\in\mathcal{P}^{\left(i\right)}\left[x_{i},y_{i}\right]}\mathcal{P}^{\prime,\left(i\right)}\left[x_{i},y_{i}\right]\cap\left[u_{i},v_{i}\right]}
\]
\[
=\prod_{i\in\theta}I_{\bigcup_{\left[u_{i},v_{i}\right]\in\mathcal{P}^{\left(i\right)}\left[x_{i},y_{i}\right]}}I_{\mathcal{P}^{\prime,\left(i\right)}\left[x_{i},y_{i}\right]\cap\left[u_{i},v_{i}\right]}
\]
\[
=\prod_{i\in\theta}I_{\bigcup_{\left[u_{i},v_{i}\right]\in\mathcal{P}^{\left(i\right)}\left[x_{i},y_{i}\right]}}\prod_{i\in\theta}I_{\mathcal{P}^{\prime,\left(i\right)}\left[x_{i},y_{i}\right]\cap\left[u_{i},v_{i}\right]}.
\]
Therefore, we can compare 
\[
|\left(I_{\mathcal{P}}-I_{\mathcal{P}^{\prime}}\right)\varXi|
\]
\[
\leq\prod_{i\in\theta}I_{\bigcup_{\left[u_{i},v_{i}\right]\in\mathcal{P}^{\left(i\right)}\left[x_{i},y_{i}\right]}}|\prod_{i\in\theta}\left(I-I_{\mathcal{P}^{\prime,\left(i\right)}\left[x_{i},y_{i}\right]\cap\left[u_{i},v_{i}\right]}\right)\varXi|,
\]
 by the maximal inequality showed above, we have 
\[
\prod_{i\in\theta}\left(I-I_{\mathcal{P}^{\prime,\left(i\right)}\left[x_{i},y_{i}\right]\cap\left[u_{i},v_{i}\right]}\right)\varXi\left(\mathbf{u},\mathbf{v}\right)|
\]
\[
\lesssim C_{\beta}2^{k}\left(\sum_{i\in\theta}\parallel\delta^{\left(i\right)}\varXi\parallel_{\beta_{i}}\right)\left(\prod_{i\in\theta}I_{\mathcal{P}^{\left(i\right)}\left[x_{i},y_{i}\right]}\right)p_{\beta-\alpha}\left(\mathbf{u},\mathbf{v}\right)q_{\alpha}\left(\mathbf{u},\mathbf{v}\right),
\]
where 
\[
\prod_{i\in\theta}I_{\mathcal{P}^{\left(i\right)}\left[x_{i},y_{i}\right]}p_{\beta-\alpha}\left(\mathbf{u},\mathbf{v}\right)q_{\alpha}\left(\mathbf{u},\mathbf{v}\right)=\sum_{\left[\mathbf{u},\mathbf{v}\right]\in\mathcal{P}\left[\mathbf{x},\mathbf{y}\right]}p_{\beta-\alpha}\left(\mathbf{u},\mathbf{v}\right)q_{\alpha}\left(\mathbf{u},\mathbf{v}\right).
\]
\[
\leq\sum_{\left[u_{1},v_{1}\right]\in\mathcal{P}^{\left(1\right)}}\cdots\sum_{\left[u_{k},v_{k}\right]\in\mathcal{P}^{\left(k\right)}}\sum_{i=1}^{k}|v_{i}-u_{i}|^{\beta_{i}}.
\]
Note that for $\beta\in\left(1,\infty\right),$we have $\sum_{\left[\mathbf{u},\mathbf{v}\right]\in\mathcal{P}\left[\mathbf{x},\mathbf{y}\right]}\sum_{i=1}^{k}|v_{i}-u_{i}|^{\beta_{i}}=\sum_{i=1}^{k}\epsilon^{\beta_{i}-1}|x_{i}-y_{i}|,$and
it follows that 
\[
\sup_{|\mathcal{P}\left[\mathbf{x},\mathbf{y}\right]|\vee|\mathcal{P}'\left[\mathbf{x},\mathbf{y}\right]|\leq\epsilon}|\int_{\mathcal{P}'\left[\mathbf{x},\mathbf{y}\right]}\varXi-\int_{\mathcal{P}\left[\mathbf{x},\mathbf{y}\right]}\varXi|\lesssim\epsilon^{\inf\beta_{i}-1},
\]
which proves our claim, and concludes the proof. 
\end{proof}
Next we will show how $\delta$ operates on functions of the form
$Y\left(\mathbf{x}\right)\square_{\mathbf{x},\mathbf{y}}^{k}X$ for
some $\mathbf{x},\mathbf{y}\in\left[0,1\right]^{k}$ which plays a
crucial role in the construction of the Young integral (as we can
consider this the local approximation of the Young integral operator).
\begin{lem}
\label{lem:Chens relation identity}Let $\theta=\left\{ 1,..,k\right\} $
and $\varXi\left(\mathbf{x},\mathbf{y}\right):=Y\left(\mathbf{x}\right)\square_{\mathbf{x},\mathbf{y}}^{k}X$
for some $X,Y\in C\left(\left[0,1\right]^{k}\right)$. For some $\mathbf{z}\in\left[\mathbf{x},\mathbf{y}\right]$
we have for all $j\in\theta$
\[
\delta_{\mathbf{z}}^{\left(j\right)}\varXi\left(\mathbf{x},\mathbf{y}\right)=-\left(Y\left(V_{j,\mathbf{z}}\mathbf{x}\right)-Y\left(\mathbf{x}\right)\right)\square_{V_{j,\mathbf{z}}\mathbf{\mathbf{x}},\mathbf{y}}^{k}X.
\]
\end{lem}

\begin{proof}
Using the commuting property $V_{i,\mathbf{x}}V_{j,\mathbf{x}}=V_{j,\mathbf{x}}V_{i,\mathbf{x}}$
for all $i,j\in\left\{ 1,..,k\right\} $ we can see that
\[
\square_{\mathbf{x},\mathbf{y}}^{k}X-\square_{\mathbf{x},\left(y_{1},..,z_{j},..,y_{k}\right)}^{k}X=\prod_{i=1}^{j-1}\left(I-V_{i,\left(y_{1},..,z_{j},..,y_{k}\right)}\right)\times\prod_{i=j+1}^{k}\left(I-V_{i,\left(y_{1},..,z_{j},..,y_{k}\right)}\right)
\]
\[
\times\left(X\left(x_{1},..,z_{j},..,x_{k}\right)-X\left(x_{1},..,y_{j},..,x_{k}\right)\right)
\]
\[
=\prod_{i=1}^{j-1}\left(I-V_{i,\mathbf{y}}\right)\times\prod_{i=j+1}^{k}\left(I-V_{i,\mathbf{y}}\right)\left(I-V_{j,\mathbf{y}}\right)X\left(x_{1},..,z_{j},..,x_{k}\right)
\]
\[
=\square_{\left(x_{1},..,z_{j},..,x_{k}\right),\mathbf{y}}^{k}X,
\]
 and we can see that
\[
\delta_{\mathbf{z}}^{\left(j\right)}Y\left(\mathbf{x}\right)\square_{\mathbf{x},\mathbf{y}}^{k}X=Y\left(\mathbf{x}\right)\square_{\mathbf{x},\mathbf{y}}^{k}X-Y\left(\mathbf{x}\right)\square_{\mathbf{x},V_{j,\mathbf{z}}\mathbf{y}}^{k}X-Y\left(V_{j,\mathbf{z}}\mathbf{\mathbf{x}}\right)\square_{V_{j,\mathbf{z}}\mathbf{x},\mathbf{y}}^{k}X
\]
\[
=-\left(Y\left(V_{j,\mathbf{z}}\mathbf{x}\right)-Y\left(\mathbf{x}\right)\right)\square_{V_{j,\mathbf{z}}\mathbf{\mathbf{x}},\mathbf{y}}^{k}X.
\]
\end{proof}
With the above Lemma we are now ready to present the an integration
theorem for Young fields on hyper-cubes $\left[0,1\right]^{k}$. 
\begin{thm}
\label{thm:Young integration}Let $Y\in\mathcal{C}^{\alpha}\left(\left[0,1\right]^{k};\mathbb{R}\right)$
and $X\in\mathcal{C}^{\beta}\left(\left[0,1\right]^{k};\mathbb{R}\right)$
for two multi-indices $\alpha\in\left(0,1\right)^{k}$ and $\beta\in\left(0,1\right)^{k}$
with the property that 
\[
\alpha_{i}+\beta_{i}>1\,\,\,\forall i\in\left\{ 1,..,k\right\} ,
\]
 and assume $\left[\mathbf{x},\mathbf{y}\right]$ is non-degenerate.
Then there exists a $C_{\alpha,\beta,k}>0$ such that 
\[
|\int_{\mathbf{x}}^{\mathbf{y}}Y\left(\mathbf{z}\right)X\left(d\mathbf{z}\right)-Y\left(\mathbf{x}\right)\square_{\mathbf{x},\mathbf{y}}^{k}X|\leq C_{\alpha,\beta,k}\parallel Y\parallel_{\alpha,+}\parallel X\parallel_{\beta,\square}q_{\beta}\left(\mathbf{x},\mathbf{y}\right)p_{\alpha}\left(\mathbf{x},\mathbf{y}\right)
\]
 with $p_{\gamma}\left(\mathbf{x},\mathbf{y}\right)=\sum_{i\in\theta}|x_{i}-y_{i}|^{\gamma_{i}}$
and $q_{\gamma}\left(\mathbf{x},\mathbf{y}\right)=\prod_{i\in\theta}|x_{i}-y_{i}|^{\gamma_{i}}$.
Moreover, it follows that
\[
\mathbf{x}\mapsto\int_{\mathbf{0}}^{\mathbf{x}}Y\left(\mathbf{z}\right)X\left(d\mathbf{z}\right)\in\mathcal{C}^{\beta}.
\]
\end{thm}

\begin{proof}
Set $\varXi\left(\mathbf{x},\mathbf{y}\right):=Y\left(\mathbf{x}\right)\square_{\mathbf{x},\mathbf{y}}X$
then we must show that for $\theta=\left\{ 1,..,k\right\} $, we have
\[
\parallel\delta^{\theta}\varXi\parallel_{\alpha+\beta}<\infty.
\]
From Lemma \ref{lem:Chens relation identity}, we have for $\mathbf{z\in\left[\mathbf{x},\mathbf{y}\right]}$,
\[
\delta_{\mathbf{z}}^{\left(j\right)}\varXi\left(\mathbf{x},\mathbf{y}\right)=-\left(Y\left(V_{j,\mathbf{z}}\mathbf{x}\right)-Y\left(\mathbf{x}\right)\right)\square_{V_{j,\mathbf{z}}\mathbf{\mathbf{x}},\mathbf{y}}^{k}X.
\]
Therefore, using the regularity of $Y\in\mathcal{C}_{+}^{\alpha}$
and $X\in\mathcal{C}_{\square}^{\beta}$ it is clear that for $\theta=\left\{ 1,..,k\right\} ,$
we have 
\[
\parallel\delta^{\theta}\varXi\parallel_{\alpha,\beta}\leq C_{\beta,\alpha}\sum_{i=1}^{k}\sup_{\left(x_{1},..,x_{i-1},x_{i+1},..,x_{k}\right)\in\left[0,1\right]^{k-1}}\parallel Y_{i,\mathbf{x}}\parallel_{\alpha_{i}}\parallel X\parallel_{\beta,\square}|y_{i}-x_{i}|^{\alpha_{i}}q_{\beta}\left(\mathbf{x},\mathbf{y}\right)
\]
\[
\leq C_{\alpha,\beta}\parallel Y\parallel_{\alpha,+}\parallel X\parallel_{\beta,\square}q_{\beta}\left(\mathbf{x},\mathbf{y}\right)p_{\alpha}\left(\mathbf{x},\mathbf{y}\right).
\]
Now the inequality and the construction of the integral follows from
Lemma \ref{thm:Sewing map}. 

For the second claim that
\[
\mathbf{x}\mapsto\int_{\mathbf{0}}^{\mathbf{x}}Y\left(\mathbf{z}\right)X\left(d\mathbf{z}\right)\in\mathcal{C}^{\beta},
\]
we note that for $\mathbf{x}\neq\mathbf{y}\in\left[0,1\right]^{k}$
(i.e $x_{i}\neq y_{i}$ $\forall i\in\left\{ 1,...,k\right\} $) we
have $\int_{\mathbf{0}}^{\mathbf{x}}=\int_{0}^{x_{1}}\cdots\int_{0}^{x_{k}}$
, and by the linearity of the integrals we get 
\[
\int_{0}^{y_{1}}\cdots\int_{0}^{y_{k-1}}\int_{0}^{y_{k}}-\int_{0}^{y_{1}}\cdots\int_{0}^{y_{k-1}}\int_{0}^{x_{k}}=\int_{0}^{y_{1}}\cdots\int_{0}^{y_{k-1}}\int_{x_{k}}^{y_{k}}.
\]
Therefore, we can use an iterative argument to see that 
\[
\square_{\mathbf{x},\mathbf{y}}^{k}\int_{\mathbf{0}}^{\cdot}Y\left(\mathbf{z}\right)X\left(\mathbf{dz}\right)=\square_{\mathbf{x'},\mathbf{y'}}^{k-1}\int_{\mathbf{0}}^{\left(\cdot,y_{k}\right)}Y\left(\mathbf{z}\right)X\left(\mathbf{dz}\right)-\square_{\mathbf{x'},\mathbf{y'}}^{k-1}\int_{\mathbf{0}}^{\left(\cdot,x_{k}\right)}Y\left(\mathbf{z}\right)X\left(\mathbf{dz}\right)
\]
\[
=\square_{\mathbf{x'},\mathbf{y'}}^{k-1}\int_{\left(\cdot,x_{k}\right)}^{\left(\cdot,y_{k}\right)}Y\left(\mathbf{z}\right)X\left(\mathbf{dz}\right)
\]
\[
\vdots
\]
\[
=\int_{\mathbf{x}}^{\mathbf{y}}Y\left(\mathbf{z}\right)X\left(d\mathbf{z}\right).
\]
From this we can see that for $Z\left(\mathbf{x}\right):=\int_{\mathbf{0}}^{\mathbf{x}}Y\left(\mathbf{z}\right)X\left(d\mathbf{z}\right)$,
we have for all $i\in\left\{ 1,...,k\right\} $
\[
|Z_{i,\mathbf{x}}\left(v_{i}\right)-Z_{i,\mathbf{x}}\left(u_{i}\right)|=|\square_{V_{i,\mathbf{v}}\mathbf{0},V_{i,\mathbf{u}}\mathbf{x}}Z|
\]
\[
\lesssim\left(\left(|Y\left(\mathbf{0}\right)|+C_{\alpha,\beta}\parallel Y\parallel_{\alpha,+}\right)\parallel X\parallel_{\beta,\square}\right)|v_{i}-u_{i}|^{\alpha_{i}},
\]
and together with the first estimate, we obtain our claim. 
\end{proof}
The above Theorem gives us a construction of the Young integral for
fields based on the generalized increment. We can also see that the
Hölder continuity of the integral is of order $\beta$ inherited by
the driving noise $X\in\mathcal{C}^{\beta}.$ This can easily be seen
by adding and subtracting $Y\left(\mathbf{x}\right)\square_{\mathbf{x},\mathbf{y}}^{k}X$
and using the above Theorem, we get 
\begin{equation}
|\int_{\mathbf{x}}^{\mathbf{y}}Y\left(\mathbf{z}\right)X\left(\mathbf{dz}\right)|\leq\left(|Y\left(\mathbf{0}\right)|+C\parallel Y\parallel_{\alpha}p_{\alpha}\left(\mathbf{x},\mathbf{y}\right)\right)\parallel X\parallel_{\beta,\square}q_{\beta}\left(\mathbf{x},\mathbf{y}\right).\label{eq:regularity of integral}
\end{equation}

\section{Differential equations driven by Young Fields }

Using the results established in the previous sections we will now
show existence and uniqueness of solutions to differential equations
driven by Hölder fields on the form 
\[
Y\left(\mathbf{x}\right)=\xi\left(\mathbf{x}\right)+\int_{\mathbf{0}}^{\mathbf{x}}f\left(Y\left(\mathbf{z}\right)\right)X\left(d\mathbf{z}\right),
\]
for some sufficiently smooth function $f$ , and $\xi,X\in\mathcal{C}^{\alpha}\left(\left[0,1\right]^{k};\mathbb{R}\right)$
with $\alpha\in\left(\frac{1}{2},1\right)^{k}$. We will construct
solutions as a fixed point iteration in the space $\mathcal{C}^{\alpha},$
\\
The function $\xi\left(\mathbf{x}\right)$ can be thought of as the
boundary behavior of $Y$. Indeed, note that the integral $\int_{\mathbf{0}}^{\mathbf{x}}$
is $0$ even when one of the variables $x_{i}$ in $\mathbf{x}$ is
$0$, and therefore when one (or more) variable(s) is $0$ (i.e. we
are one the boundary of the hyper-cube $\left[0,1\right]^{k}$), then
$Y$ is completely determined by $\xi$. When $k=1$, it would be
sufficient to describe the initial value of the solution (as a point
in the $1-$dim domain). However, when considering hyper-cubes, it
is usually necessary to provide more information about the behavior
of the solution on the boundary, and therefore $\xi$ is meant to
capture just this behavior.
\begin{thm}
\label{thm:Ambit differential equation}Let $X\in\mathcal{C}^{\alpha}\left(\left[0,1\right]^{k};\mathbb{R}\right)$
with $\alpha\in\left(\frac{1}{2},1\right){}^{k}$, $\xi_{0}\in\mathcal{C}^{\alpha}\left(\left[0,1\right]^{k};\mathbb{R}\right)$
with $\xi_{0}\left(\mathbf{0}\right)\in\mathbb{R}$ and $f\in C_{b}^{2}\left(\mathbb{R};\mathbb{R}\right)$.
There exists a unique solution in $\mathcal{C}^{\alpha}\left(\left[0,1\right]^{k};\mathbb{R}\right)$
to the rough field equation 
\begin{equation}
Y\left(\mathbf{x}\right)=\xi_{0}\left(\mathbf{x}\right)+\int_{\mathbf{0}}^{\mathbf{x}}f\left(Y\left(\mathbf{z}\right)\right)X\left(d\mathbf{z}\right),\label{eq:Ambit integral equation}
\end{equation}
where the integral is understood as a rough field integral in terms
of Theorem \ref{thm:Young integration}.
\end{thm}

\begin{proof}
We will use techniques from \cite{FriHai}, based on proving a unique
fixed point of a Picard type solution map in the space $\mathcal{C}^{\beta}\left(\left[0,1\right]^{k};\mathbb{R}\right)$
for $\beta\in\left(\mathbf{\frac{1}{2}},\alpha\right)\subset\left(0,1\right)^{k}$.
After this, we show that the solution is actually in the the space
$\mathcal{C}^{\alpha}$ as a consequence of the regularity of the
driving field $X\in\mathcal{C}^{\alpha}\left(\left[0,1\right]^{k}\right)$.
The proof will be done in two parts and is based on the behavior of
a linear solution map on the space $\mathcal{C}^{\beta}\left(\left[\rho,\rho+\mathbf{T}\right];\mathbb{R}\right)$
for some $\mathbf{T}>\mathbf{0}$ and $\rho\in\left[\mathbf{0},\mathbf{1}-\mathbf{T}\right]$.
When considering the space $\mathcal{C}^{\beta}\left(\left[\rho,\rho+\mathbf{T}\right];\mathbb{R}\right)$
it is with respect to $X$ also restricted to this space, i.e. $X\in\mathcal{C}^{\beta}\left(\left[\rho,\rho+\mathbf{T}\right];\mathbb{R}\right)$. 

First we will show that the solution map is invariant on a properly
constructed unit ball in $\mathcal{C}^{\beta}\left(\left[\rho,\rho+\mathbf{T}\right];\mathbb{R}\right)$,
and then we will show that it is a contraction mapping on $\mathcal{C}^{\beta}\left(\left[\rho,\rho+\mathbf{T}\right];\mathbb{R}\right)$.
Both of these properties will depend on the choice of a small size
of $\mathbf{T}>\text{\ensuremath{\mathbf{0}}}$ . Therefore, our second
objective is to extend the solution from the hyper cube $\left[\mathbf{\rho},\rho+\mathbf{T}\right]$
to all of $\left[0,1\right]^{k}$ by gluing together the solutions
on all the sets $\left[\mathbf{\rho},\rho+\mathbf{T}\right]$.

Consider now some fixed $\mathbf{T}>\mathbf{0}$, $\rho\in\left[\mathbf{0},\mathbf{1}-\mathbf{T}\right]$,
and a hyper cube $\left[\rho,\rho+\mathbf{T}\right]$. We can reformulate
equation (\ref{eq:Ambit integral equation}) to make the proof a bit
simpler and consider $\xi_{\rho}\in\mathcal{C}^{\beta}\left(\left[\rho,\rho+\mathbf{T}\right]\right)$
which satisfy for some $M\in\mathbb{R}$
\[
\sup_{\rho\in\left[\mathbf{0},\mathbf{1}-\mathbf{T}\right]}\parallel\xi_{\rho}\parallel_{\beta;\left[\mathbf{\rho,\rho+\mathbf{T}}\right]}\leq M\parallel\xi_{0}\parallel_{\beta;\left[0,1\right]^{k}}.
\]
 Define a solution mapping 
\[
\mathcal{V}_{\mathbf{T}}^{\rho}:\mathcal{C}^{\beta}\left(\left[\rho,\rho+\mathbf{T}\right]\right)\rightarrow\mathcal{C}^{\beta}\left(\left[\rho,\rho+\mathbf{T}\right]\right)
\]
\[
Y\mapsto\mathcal{V}_{\mathbf{T}}^{\rho}\left(Y\right):=\xi_{\rho}\left(\mathbf{\cdot}\right)+\int_{\rho}^{\mathbf{\cdot}}f\left(Y\left(\mathbf{z}\right)\right)X\left(d\mathbf{z}\right),
\]
where the integrals runs from $\rho$ as the initial value of the
set $\left[\rho,\rho+\mathbf{T}\right]$. The fact that this mapping
indeed is in $\mathcal{C}_{\square}^{\beta}$ is a consequence of
Theorem \ref{thm:Young integration}, and the fact that for $i\in\left\{ 1,...,k\right\} $
\begin{equation}
\begin{array}{c}
|f\left(Y_{i,\mathbf{x}}\left(y_{i}\right)+\xi_{\rho,i,\mathbf{x}}\left(y_{i}\right)\right)-f\left(Y_{i,\mathbf{x}}\left(x_{i}\right)+\xi_{\rho,i,\mathbf{x}}\left(x_{i}\right)\right)|\\
\lesssim|f|_{C_{b}^{1}}\left(Y_{i,\mathbf{x}}\left(y_{i}\right)-Y_{i,\mathbf{x}}\left(x_{i}\right)+\xi_{\rho,i,\mathbf{x}}\left(y_{i}\right)-\xi_{\rho,i,\mathbf{x}}\left(x_{i}\right)\right)|\\
\lesssim|f|_{C_{b}^{1}}\left(\parallel Y\parallel_{\beta}+\parallel\xi_{\rho}\parallel_{\beta}\right)|y_{i}-x_{i}|^{\beta_{i}}.
\end{array}\label{eq:functions of h=0000F6lder path}
\end{equation}
 Without loss of generality, in the rest of the proof we consider
a subspace of $\mathcal{C}^{\beta}\left(\left[\rho,\rho+\mathbf{T}\right]\right)$
of all fields $Y$ which satisfy $Y\left(\rho\right)=\xi_{\rho}\left(\rho\right)$.
Next, we construct a unit ball in $\mathcal{C}^{\beta}$ centered
at $0$ defined by 
\[
\mathcal{B}_{\rho,\mathbf{T}}:=\left\{ Y\in\mathcal{C}^{\beta}\left(\left[\rho,\rho+\mathbf{T}\right]\right)|\parallel Y\parallel_{\beta}\leq1\right\} ,
\]
and we want to show that for $Y\in\mathcal{B}_{\rho,\mathbf{T}}$
we can choose $\mathbf{T}$ such that $\mathcal{V}_{\mathbf{T}}\left(Y\right)\in\mathcal{B}_{\rho,\mathbf{T}}$.
However, from Theorem \ref{thm:Young integration} together with Equation
(\ref{eq:functions of h=0000F6lder path}), we know 
\[
|\square_{\mathbf{x},\mathbf{y}}\xi_{\rho}+\int_{\mathbf{x}}^{\mathbf{\cdot}\mathbf{y}}f\left(Y\left(\mathbf{z}\right)\right)X\left(d\mathbf{z}\right)|
\]
\begin{equation}
\leq\left(\parallel\xi_{\rho}\parallel_{\beta,\square}+\left(|Y|_{\infty}+\parallel Y\parallel_{\beta}\right)\parallel X\parallel_{\beta,\square}p_{\beta}\left(\mathbf{x},\mathbf{y}\right)\right)q_{\beta}\left(\mathbf{x},\mathbf{y}\right)\label{eq:square reg invariance}
\end{equation}
Therefore we see that 
\begin{equation}
\parallel\mathcal{V}_{\mathbf{T}}\left(Y\right)\parallel_{\beta}\lesssim\left(\parallel\xi_{\rho}\parallel_{\alpha}+\left(|\xi_{\mathbf{0}}|+\parallel Y\parallel_{\beta}\right)\parallel X\parallel_{\alpha,\square}\right)p_{\alpha-\beta}\left(\mathbf{0},\mathbf{T}\right),\label{eq:boundedness of solution in B}
\end{equation}
where we have used that $\parallel X\parallel_{\beta,\square}\leq\parallel X\parallel_{\alpha,\square}q_{\alpha-\beta}\left(\mathbf{0},\mathbf{T}\right)$
and that 
\[
\parallel\mathcal{V}_{\mathbf{T}}\left(Y\right)\parallel_{\beta}=\sum_{i=1}^{k}\sup_{\mathbf{x}\in\left[\rho,\rho+\mathbf{T}\right]}\parallel\mathcal{V}_{\mathbf{T}}\left(Y\right)_{i,\mathbf{x}}\parallel_{\beta_{i}}+\parallel\mathcal{V}_{\mathbf{T}}\left(Y\right)\parallel_{\beta,\square},
\]
then for 
\[
|\mathcal{V}_{\mathbf{T}}\left(Y\right)_{i,\mathbf{x}}\left(v\right)-\mathcal{V}_{\mathbf{T}}\left(Y\right)_{i,\mathbf{x}}\left(u\right)|=|\xi_{\rho,i,\mathbf{x}}\left(v\right)-\xi_{\rho,i,\mathbf{x}}\left(u\right)+\int_{V_{i\mathbf{u}}\mathbf{0}}^{V_{i,\mathbf{v}}\mathbf{x}}f\left(Y\left(\mathbf{z}\right)\right)X\left(d\mathbf{z}\right)|,
\]
we can also use inequality (\ref{eq:square reg invariance}). From
this estimate, it is clear that we can choose $\mathbf{T}=\mathbf{T}_{0}$
small enough such that $p_{\alpha-\beta}\left(\mathbf{0},\mathbf{T}_{0}\right)$
(notice that $\mathbf{T}_{0}$ is independent of $\rho$) gets sufficiently
small to obtain the requirement
\[
\parallel\mathcal{V}_{\mathbf{T}_{0}}\left(Y\right)\parallel_{\beta;\left[\rho,\rho+\mathbf{T}_{0}\right]}\leq1,
\]
and we can conclude that $\mathcal{V}_{\mathbf{T}_{0}}^{\rho}$ leaves
$\mathbf{\mathcal{B}_{\rho,\mathbf{T}_{0}}}$ invariant. 

Next, we will show that $\mathcal{V}_{\mathbf{T}}$ is a contraction
mapping on $\mathcal{C}^{\beta}\left(\left[\rho,\rho+\mathbf{T}\right]\right).$
We can easily see that for two elements $Y,Z\in\mathcal{C}^{\beta}$,
both with boundary given by $\xi_{\rho}\in\mathcal{C}^{\alpha}$ and
satisfying 
\[
\parallel Y\parallel_{\beta},\parallel Z\parallel_{\beta}\leq M,
\]
we use Theorem \ref{thm:Young integration} to see that
\[
\parallel\mathcal{V}_{\mathbf{T}}\left(Y\right)-\mathcal{V}_{\mathbf{T}}\left(Z\right)\parallel_{\beta,\square}
\]
\[
=\parallel\int_{\rho}^{\cdot}f\left(Y\left(\mathbf{z}\right)\right)X\left(d\mathbf{z}\right)-\int_{\rho}^{\cdot}f\left(Z\left(\mathbf{z}\right)\right)X\left(d\mathbf{z}\right)\parallel_{\beta,\square}
\]
\[
\leq C\parallel f\left(Y\right)-f\left(Z\right)\parallel_{\beta,+}\parallel X\parallel_{\alpha,\square}q_{\alpha-\beta}\left(\mathbf{0},\mathbf{T}\right).
\]
We must therefore show that 
\begin{equation}
\parallel f\left(Y\right)-f\left(Z\right)\parallel_{\beta,+}\lesssim\parallel Y-Z\parallel_{\beta},\label{eq:function of young field bound}
\end{equation}
but this is straight forward, using Lemma 8.2 in \cite{FriHai}, we
know that in the 
\[
\parallel f\left(Y_{i,\mathbf{x}}\right)-f\left(Z_{i,\mathbf{x}}\right)\parallel_{\beta_{i}}\leq C_{\beta,M}\parallel f\parallel_{C_{b}^{2}}\left(\parallel Y_{i,\mathbf{x}}-Z_{i,\mathbf{x}}\parallel_{\beta_{i}}\right),
\]
and we can conclude that Equation (\ref{eq:function of young field bound})
holds. Therefore, 
\[
\parallel\mathcal{V}_{\mathbf{T}}\left(Y\right)-\mathcal{V}_{\mathbf{T}}\left(Z\right)\parallel_{\beta,\square}\leq C_{\beta,M}\parallel f\parallel_{C_{b}^{2}}\parallel Y-Z\parallel_{\beta}\parallel X\parallel_{\alpha,\square}q_{\alpha-\beta}\left(\mathbf{0},\mathbf{T}\right),
\]
and we can find that 
\begin{equation}
\parallel\mathcal{V}_{\mathbf{T}}\left(Y\right)-\mathcal{V}_{\mathbf{T}}\left(Z\right)\parallel_{\beta}\leq\tilde{C}_{\beta,M}\parallel f\parallel_{C_{b}^{2}}\parallel Y-Z\parallel_{\beta}\parallel X\parallel_{\alpha,\square}p_{\alpha-\beta}\left(\mathbf{0},\mathbf{T}\right).\label{eq:contraction}
\end{equation}
Again, it is clear that we can choose $\mathbf{T}$ small enough such
that 
\[
\parallel\mathcal{V}_{\mathbf{T}}\left(Y\right)-\mathcal{V}_{\mathbf{T}}\left(Z\right)\parallel_{\beta}\leq q\parallel Y-Z\parallel_{\beta}
\]
for $q\in\left(0,1\right)$. We can conclude that $\mathcal{V}_{\mathbf{T}_{0}}$
admits a unique fixed point in $\mathcal{B}_{\rho,\mathbf{T}_{0}}$
which is the solution to equation (\ref{eq:Ambit integral equation}).
We can now consider Equation (\ref{eq:boundedness of solution in B}),
and choose $\mathbf{\tilde{T}}_{0}$ small, such that $\left(1-\parallel X\parallel_{\alpha,\square}p_{\alpha-\beta}\left(\mathbf{0},\mathbf{\tilde{T}}_{0}\right)\right)>0,$
and see that the solution $Y\in\mathcal{C}^{\beta}$ satisfy 
\[
\parallel Y\parallel_{\beta}\left(1-\parallel X\parallel_{\alpha,\square}p_{\alpha-\beta}\left(\mathbf{0},\mathbf{T}\right)\right)\leq C_{X}\left(\parallel\xi_{\rho}\parallel_{\alpha}+|\xi_{\mathbf{0}}|\right),
\]
and we see that on $\left[\rho,\rho+\mathbf{\tilde{T}}_{0}\vee\mathbf{T}_{0}\right]$
we have 
\begin{equation}
\parallel Y\parallel_{\beta}\leq C_{X}\left(\parallel\xi_{\rho}\parallel_{\alpha}+|\xi_{\mathbf{0}}|\right)\leq MC_{X}\left(\parallel\xi_{0}\parallel_{\alpha}+|\xi_{\mathbf{0}}|\right),\label{eq:bound of solution}
\end{equation}
where we used the assumption that $\parallel\xi_{\rho}\parallel_{\alpha}\leq M\parallel\xi_{0}\parallel_{\alpha}$
uniformly in $\rho$. 

The second part of the proof is to extend the solution from $\mathcal{C}^{\beta}\left(\left[\rho,\rho+\mathbf{T}\right]\right)$
to $\mathcal{C}^{\beta}\left(\left[0,1\right]^{k}\right).$ First
notice that the interval $\left[\rho,\rho+\mathbf{T}\right]$ was
created for some arbitrary point $\rho\in\left[\mathbf{0},\mathbf{1}-\mathbf{T}\right]$,
and specially there exists a solution to equation (\ref{eq:Ambit integral equation})
on $\left[\mathbf{0},\mathbf{T}\right]$. From classical theory of
differential equations, proofs of existence an uniqueness of solutions
on an interval $\left[0,1\right]$ is based on first proving existence
on $\left[0,T\right]$ for some small $T>0$ and then iterate this
solution on $\left[0,T\right]$, $\left[T,2T\right]...,\left[kT,1\right]$
by choosing the initial value of the solution on each interval to
be the terminal value on the interval prior. However, when we work
on hyper cubes, in general, we need to give information about the
solution on the boundary of the hyper cube we work on to be able to
make sure that the solution is of sufficient continuity when gluing
together the cubes in the end. Notice that the hyper cube $\left[\mathbf{0},\mathbf{T}\right]$
is one of $2^{k}$ subsets of $\left[\mathbf{0},2\mathbf{T}\right]$
divided in $\mathbf{T}$. When we want to iterate the solution from
$\left[\mathbf{0},\mathbf{T}\right]$ to one of these $2^{k}$ adjoining
hyper cubes, we must specify the choice of $\xi$ to make sure that
the adjoining boundaries of these hyper cubes have the same values.
We know that the $2^{k}$ hyper cube subsets of $\left[\mathbf{0},2\mathbf{T}\right]$
divided in $\mathbf{T}$ are of the form 
\[
\left[\rho,\rho+\mathbf{T}\right]=\left[0,T_{1}\right]\times\left[T_{2},2T_{2}\right]\times..\times\left[T_{k-1},2T_{k-1}\right]\times\left[0,T_{k}\right],
\]
for $\rho=\left(0,T_{2},..,T_{k-1},0\right).$ First, we want to extend
the solution from $\left[\mathbf{0},\mathbf{T}\right]$ to $\left[\gamma^{\left(1\right)},\gamma^{\left(1\right)}+\mathbf{T}\right]$
for $\gamma^{\left(1\right)}=\left(0,..,0,T_{i},0,..,0\right)$ for
some $0\leq i\leq k$ (the number $1$ in the superscript signifies
that it is one non-zero entry in $\gamma^{\left(1\right)}$) and we
must consider the equation 
\[
Y\left(\mathbf{x}\right)=\xi_{0}\left(\mathbf{x}\right)+\int_{\mathbf{0}}^{\mathbf{x}}f\left(Y\left(\mathbf{z}\right)\right)X\left(d\mathbf{z}\right),\,\,\,\mathbf{x}\in\left[\gamma^{\left(1\right)},\gamma^{\left(1\right)}+\mathbf{T}\right]
\]
and create a suitable solution map with respect to this equation,
similar to what we have done earlier for general $\left[\rho,\rho+\mathbf{T}\right]$.
Note that the integral $\int_{\mathbf{0}}^{\mathbf{x}}$ can be split
in the $i-$Th component of an integral $\int_{0}^{x_{i}}=\int_{0}^{T_{i}}+\int_{T_{i}}^{x_{i}}$,
and obtain the equation 
\[
Y^{|\left[\gamma^{\left(1\right)},\gamma^{\left(1\right)}+\mathbf{T}\right]}\left(\mathbf{x}\right)=\xi_{0}\left(\mathbf{x}\right)+\int_{0}^{x_{1}}...\int_{0}^{T_{i}}...\int_{0}^{x_{k}}f\left(Y\left(\mathbf{z}\right)\right)X\left(d\mathbf{z}\right)
\]
\[
+\int_{0}^{x_{1}}...\int_{T_{i}}^{x_{i}}...\int_{0}^{x_{k}}f\left(Y\left(\mathbf{z}\right)\right)X\left(d\mathbf{z}\right),
\]
\begin{equation}
=\left(\xi_{0}\left(\mathbf{x}\right)-\xi_{0}\left(x_{1},..,T_{i},..,x_{k}\right)\right)+Y^{|\left[\mathbf{0},\mathbf{T}\right]}\left(x_{1},..,T_{i},..,x_{k}\right)+\int_{\gamma^{\left(1\right)}}^{\mathbf{x}}f\left(Y\left(\mathbf{z}\right)\right)X\left(d\mathbf{z}\right)\label{eq:division of equation with one point}
\end{equation}
where $Y^{|\left[\mathbf{0},\mathbf{T}\right]}$ is the solution to
equation (\ref{eq:Ambit integral equation}) on $\left[\mathbf{0},\mathbf{T}\right]$.
Define 
\[
\xi_{\gamma^{\left(1\right)}}\left(\mathbf{x}\right):=\left(\xi_{0}\left(\mathbf{x}\right)-\xi_{0}\left(V_{i,\gamma^{\left(1\right)}}\mathbf{x}\right)\right)+Y^{|\left[\mathbf{0},\mathbf{T}\right]}\left(V_{i,\gamma^{\left(1\right)}}\mathbf{x}\right).
\]
To prove existence and uniqueness, of $Y^{|\left[\gamma^{\left(1\right)},\gamma^{\left(1\right)}+\mathbf{T}\right]}\left(\mathbf{x}\right)$
in the above equation, we know it is sufficient to prove that $\xi_{\gamma^{\left(1\right)}}\in\mathcal{C}^{\alpha}\left(\left[\gamma^{\left(1\right)},\gamma^{\left(1\right)}+\mathbf{T}\right];\mathbb{R}\right)$,
and that $\parallel\xi_{\gamma^{\left(1\right)}}\parallel_{\beta}\leq M\parallel\xi_{0}\parallel_{\beta}$,
and then apply the general result for existence and uniqueness on
general domains $\left[\gamma^{\left(1\right)},\gamma^{\left(1\right)}+\mathbf{T}\right]$
that we obtained in the beginning of this proof. We can see that $\xi_{0}\in\mathcal{C}^{\beta}\left(\left[\gamma^{\left(1\right)},\gamma^{\left(1\right)}+\mathbf{T}\right];\mathbb{R}\right)$
as it is in $\mathcal{C}^{\beta}\left(\left[0,1\right]^{k};\mathbb{R}\right)$
and it is evident that 
\[
\mathbf{x}\mapsto Y^{|\left[\mathbf{0},\mathbf{T}\right]}\left(V_{i,\gamma^{\left(1\right)}}\mathbf{x}\right)\in\mathcal{C}^{\beta}\left(\left[\gamma^{\left(1\right)},\gamma^{\left(1\right)}+\mathbf{T}\right];\mathbb{R}\right).
\]
Indeed, notice that $\mathbf{x}\mapsto Y^{|\left[\mathbf{0},\mathbf{T}\right]}\left(V_{i,\gamma^{\left(1\right)}}\mathbf{x}\right)$
is constant in variable $i$. Moreover, we have that for all $1\leq i\leq k$
\[
\parallel Y^{|\left[\mathbf{0},\mathbf{T}\right]}\left(V_{i,\gamma^{\left(1\right)}}\cdot\right)\parallel_{\beta;\left[\gamma^{\left(1\right)},\gamma^{\left(1\right)}+\mathbf{T}\right]}
\]
\[
\leq\parallel Y^{|\left[\mathbf{0},\mathbf{T}\right]}\left(\cdot\right)\parallel_{\beta;\left[\mathbf{0},\mathbf{T}\right]}
\]
\[
\leq M\parallel\xi_{0},\xi_{0}^{\prime}\parallel_{\beta;\left[0,1\right]^{k}}
\]
where the last bound is obtained from Equation (\ref{eq:bound of solution}).
It follows that there exists a solution 
\[
\mathbf{x}\mapsto Y^{|\left[\gamma^{\left(1\right)},\gamma^{\left(1\right)}+\mathbf{T}\right]}\left(\mathbf{x}\right)
\]
 on $\left[\gamma^{\left(1\right)},\gamma^{\left(1\right)}+\mathbf{T}\right]$
which is connected to the solution $Y^{|\left[\mathbf{0},\mathbf{T}\right]}$
on the boundary. The result holds for all $\gamma^{\left(1\right)}=\left(0,..,T_{i},..,0\right)$,
i.e. all $i\in\left\{ 1,..,k\right\} $.

More generally, consider the operator $D_{\mathbf{z}}^{\theta}=\prod_{i\in\theta}\left(V_{i,\mathbf{x},\mathbf{z}}+V_{i,\mathbf{z},\mathbf{y}}\right)$
which is dividing the interval $\left[\mathbf{x},\mathbf{y}\right]$
in $\mathbf{z}$ and considers a linear evaluation of a function on
this division. Consider the interval $\left[\mathbf{0},\mathbf{x}\right]$
for $\mathbf{x}\in\left[\gamma^{\left(1\right)},\gamma^{\left(1\right)}+\mathbf{T}\right],$
the point $\gamma^{\left(1\right)}$ divides $\left[\mathbf{0},\mathbf{x}\right]$
into two hyper cubes, as seen in Equation (\ref{eq:division of equation with one point}).
From now on we will use the division of hyper cubes induced by $D$.
Next we assume that $\gamma^{\left(2\right)}=\left(0,..0,T_{i},0,..,0,T_{j},0,..,0\right)$
i.e. $\gamma^{\left(2\right)}$ contains $2$ non zero entries with
$T_{i}$ and $T_{j}$ for $i\neq j$. We then want to construct the
solution to, 
\[
Y\left(\mathbf{x}\right)=\xi_{0}\left(\mathbf{x}\right)+\int_{\mathbf{0}}^{\mathbf{x}}f\left(Y\left(\mathbf{z}\right)\right)X\left(d\mathbf{z}\right),\,\,\,\mathbf{x}\in\left[\gamma^{\left(2\right)},\gamma^{\left(2\right)}+\mathbf{T}\right],
\]
again using $D_{\mathbf{z}}^{\theta}$ with $\theta=\left(i,j\right)$
and the linearity of the integral, we write 
\[
Y\left(\mathbf{x}\right)=\xi_{0}\left(\mathbf{x}\right)+D_{\gamma^{\left(2\right)}}^{\theta}\left(\int_{\mathbf{0}}^{\mathbf{x}}f\left(Y\left(\mathbf{z}\right)\right)X\left(d\mathbf{z}\right)\right)
\]
\[
=\xi_{0}\left(\mathbf{x}\right)+\left(V_{i,\mathbf{0},\gamma^{\left(2\right)}}+V_{i,\gamma^{\left(2\right)},\mathbf{x}}\right)\left(V_{j,\mathbf{0},\gamma^{\left(2\right)}}+V_{j,\gamma^{\left(2\right)},\mathbf{x}}\right)\left(\int_{\mathbf{0}}^{\mathbf{x}}f\left(Y\left(\mathbf{z}\right)\right)X\left(d\mathbf{z}\right)\right)
\]
The integral $\int_{\mathbf{0}}^{\mathbf{x}}$ is divided into $4$
parts, one of them is depending on the solution $Y^{|\left[\mathbf{0},\mathbf{T}\right]}\left(\cdot,..,T_{i},..,T_{j},..,\cdot\right)$
(i.e.. with two fixed variables), and there are two integrals associated
to the solutions $Y^{|\left[V_{i,\gamma^{\left(2\right)}}\mathbf{0},V_{i,\gamma^{\left(2\right)}}\mathbf{x}\right]}\left(\cdot,..,T_{j},..,\cdot\right)$
and $Y^{|\left[V_{j,\gamma^{\left(2\right)}}\mathbf{0},V_{j,\gamma^{\left(2\right)}}\mathbf{x}\right]}\left(\cdot,..,T_{i},..,\cdot\right)$
both already constructed as solutions in the prior step (showing solutions
on 
\[
\left[\gamma^{\left(1\right)},\gamma^{\left(1\right)}+\mathbf{T}\right]=\left[0,T_{1}\right]\times..\times\left[T_{i},2T_{i}\right]\times...\times\left[0,T_{k}\right]
\]
 for all $1\leq i\leq k$. ). The forth is the integral with respect
to the the hyper cube $\left[\gamma^{\left(2\right)},\gamma^{\left(2\right)}+\mathbf{T}\right]$.
Similarly as earlier, it is straight forward to show that the conditions
on $Y^{|\left[V_{i,\gamma^{\left(2\right)}}\mathbf{0},V_{i,\gamma^{\left(2\right)}}\mathbf{x}\right]}\left(\cdot,..,T_{j},..,\cdot\right),\,\,Y\left(\cdot,..,T_{i},..,\cdot\right)$
and $Y^{|\left[\mathbf{0},\mathbf{T}\right]}\left(\cdot,..,T_{i},..,T_{j},..,\cdot\right)$
are fulfilled to conclude that there exists a unique solution to Equation
($\ref{eq:Ambit integral equation})$ on $\left[\gamma^{\left(2\right)},\gamma^{\left(2\right)}+\mathbf{T}\right]$
when $\gamma^{\left(2\right)}$ contains two non-zero entries of the
form $\left(T_{i},T_{j}\right)$, and we know that the solution on
this domain is identical on its boundary as the solutions on the domains
we have constructed prior. \\
\\
By induction principles, we assume there exists unique solutions with
boundaries identical to the prior constructed solutions on all intervals
of the form $\left[\gamma^{\left(i-1\right)},\gamma^{\left(i-1\right)}+\mathbf{T}\right]$
where $\gamma^{(i-1)}$ contains $i-1$ non zero entries each chosen
in order from $\left(T_{1},..,T_{k}\right)$. That is, at position
$n$ in $\gamma^{(i-1)}$ it is either $0$ or $T_{n}$ for any $1\leq n\leq k$.
Then we want to prove that the solution exists for any hyper cube
$\left[\gamma^{\left(i\right)},\gamma^{(i)}+\mathbf{T}\right]$ where
$\gamma^{\left(i\right)}$ consist of $i$ ordered and unique non-zero
entries as described for $\gamma^{(i-1)}$. Consider $\theta\in\left\{ 1,..,k\right\} ^{i}$
of $i$ ordered components, i.e. if $i\leq j$ then $\theta$ contains
no entries $\left(..,j.,.,i,..\right).$ For $\mathbf{x}\in\left[\gamma^{(i)},\gamma^{(i)}+\mathbf{T}\right]$
we want to divide $\left[\mathbf{0},\mathbf{x}\right]$ in $\gamma^{\left(i\right)}$
according to $\theta$ in the sense that 
\[
D_{\gamma^{\left(i\right)}}^{\theta}\left(\int_{\mathbf{0}}^{\mathbf{x}}f\left(Y\left(\mathbf{z}\right)\right)X\left(d\mathbf{z}\right)\right)=\prod_{l\in\theta}\left(V_{l,\mathbf{x},\mathbf{z}}+V_{l,\mathbf{z},\mathbf{y}}\right)\int_{\mathbf{0}}^{\mathbf{x}}f\left(Y\left(\mathbf{z}\right)\right)X\left(d\mathbf{z}\right),
\]
 which would yield a sum of $2^{i}$ integrals. One of these integrals
is an integral over $\left[\gamma^{\left(i\right)},\gamma^{\left(i\right)}+\mathbf{T}\right],$and
another is over $\left[\mathbf{0},\gamma^{\left(i\right)}\right]$,
which would represent the solution on $\left[\mathbf{0},\mathbf{T}\right]$
with $i$ fixed variables $\left\{ T_{l}\right\} _{l\in\theta}$ $.$
All other integrals represent solutions on domains $\left[\gamma^{\left(n\right)},\gamma^{\left(n\right)}+\mathbf{T}\right]$
for $1<n<i$ with $i-n$ fixed variables. All of these solutions on
domains on the form $\left[\gamma^{\left(n\right)},\gamma^{\left(n\right)}+\mathbf{T}\right]$
for $1\leq n<i$ are already proven to exist uniquely (by induction
assumption ) and are connected on each of their boundaries, and can
therefore be collected in a term called $\xi_{\gamma^{\left(i\right)}}.$
Similar to what we have shown earlier for $\gamma^{\left(1\right)}$
and $\gamma^{\left(2\right)}$, it follows that $\xi_{\gamma^{\left(i\right)}}\in\mathcal{C}^{\beta}\left(\left[\gamma^{\left(i\right)},\gamma^{\left(i\right)}+\mathbf{T}\right];\mathbb{R}\right)$,
and we can show that $\parallel\xi_{\gamma^{\left(i\right)}}\parallel_{\beta}\leq M\parallel\xi_{0}\parallel_{\beta}$.
Therefore, we conclude that there exists a unique solution to Equation
(\ref{eq:Ambit integral equation}) on any domain $\left[\gamma^{\left(i\right)},\gamma^{\left(i\right)}+\mathbf{T}\right]$
for all $1\leq i\leq k$ . In particular there exists solutions on
all hyper-cubes on the form 
\[
\left[0,T_{1}\right]\times..\times\left[T_{j},2T_{j}\right]\times...\times\left[T_{k},2T_{k}\right],
\]
 and all the solutions is connected on its boundary to the other solutions.
This leads to the conclusion that there exists a unique solution to
Equation (\ref{eq:Ambit integral equation}) on all $2^{k}$ hyper
cubes which is subsets of $\left[\mathbf{0},2\mathbf{T}\right]$ divided
in $\mathbf{T}$, and by Proposition \ref{prop:scaleability of H=0000F6lder norms},
we conclude that there exists a unique solution to Equation (\ref{eq:Ambit integral equation})
on $\left[\mathbf{0},2\mathbf{T}\right]$. We then iterate this procedure
to obtain existence and uniqueness on all of $\left[0,1\right]^{k}.$
\end{proof}
\begin{cor}
Let $X,\tilde{X}\in\mathcal{C}^{\alpha}\left(\left[0,1\right]^{k};\mathbb{R}\right)$
be two Hölder fields of regularity $\alpha\in\left(\frac{1}{2},1\right)^{k}$.
Assume $Y,\tilde{Y}\in\mathcal{C}^{\alpha}$ are two solutions to
the correpsonding equations 
\[
\begin{array}{c}
Y\left(\mathbf{x}\right)=\xi\left(\mathbf{x}\right)+\int_{\mathbf{0}}^{\mathbf{x}}f\left(Y\left(\mathbf{z}\right)\right)X\left(d\mathbf{z}\right)\\
\tilde{Y}\left(\mathbf{x}\right)=\tilde{\xi}\left(\mathbf{x}\right)+\int_{\mathbf{0}}^{\mathbf{x}}f\left(\tilde{Y}\left(\mathbf{z}\right)\right)\tilde{X}\left(d\mathbf{z}\right)
\end{array}
\]
 for $\xi,\tilde{\xi}\in\mathcal{C}^{\alpha}$ and $f\in C_{b}^{2}\left(\mathbb{R};\mathbb{R}\right)$.
Choose $M\in\mathbb{R}$ such that 
\[
\parallel X\parallel_{\alpha}\vee\parallel\tilde{X}\parallel_{\alpha}\vee\parallel Y\parallel_{\alpha}\vee\parallel\tilde{Y}\parallel_{\alpha}\leq M.
\]
Then for all $\beta<\alpha\in\left(\frac{1}{2},1\right)^{k}$ we have
\[
\parallel Y-\tilde{Y}\parallel_{\beta}\leq C_{M,f}\left(|\xi\left(\mathbf{0}\right)-\tilde{\xi}\left(\mathbf{0}\right)|+\parallel\xi-\tilde{\xi}\parallel_{\beta}+\parallel X-\tilde{X}\parallel_{\beta,\square}\right).
\]
 
\end{cor}

\begin{proof}
Again consider the solution on a hyper-cube $\left[\rho,\rho+\mathbf{T}\right]$.
Note that $Y$ is the fixed point of a solution map 
\[
Z\left(\mathbf{x}\right):=\mathcal{V}_{\mathbf{T}}\left(Y\right)\left(\mathbf{x}\right)=\left(\xi\left(\mathbf{x}\right)+\int_{0}^{\mathbf{x}}f\left(Y\left(\mathbf{z}\right)\right)X\left(d\mathbf{z}\right)\right)
\]
 similar to that constructed in the proof of Theorem \ref{thm:Ambit differential equation},
and we have the same for $\tilde{Y}$ and $\tilde{\xi}$ in $\mathcal{C}^{\alpha}$
Moreover, it is clear that 
\[
\parallel Y-\tilde{Y}\parallel_{\beta}\leq\parallel\xi-\tilde{\xi}\parallel_{\beta}+\parallel\int_{0}^{\mathbf{\cdot}}f\left(Y\left(\mathbf{z}\right)\right)X\left(d\mathbf{z}\right)-\int_{0}^{\mathbf{\cdot}}f\left(\tilde{Y}\left(\mathbf{z}\right)\right)\tilde{X}\left(d\mathbf{z}\right)\parallel_{\beta}.
\]
From Theorem \ref{thm:Young integration} we have that 
\[
|\square_{\mathbf{x},\mathbf{y}}\int_{0}^{\mathbf{\cdot}}f\left(Y\left(\mathbf{z}\right)\right)X\left(d\mathbf{z}\right)-\square_{\mathbf{x},\mathbf{y}}\int_{0}^{\mathbf{\cdot}}f\left(\tilde{Y}\left(\mathbf{z}\right)\right)\tilde{X}\left(d\mathbf{z}\right)|
\]
\[
\leq|f\left(Y\left(\mathbf{x}\right)\right)\square_{\mathbf{x},\mathbf{y}}X-f\left(\tilde{Y}\left(\mathbf{x}\right)\right)\square_{\mathbf{x},\mathbf{y}}\tilde{X}|
\]
\[
+\sum_{i=1}^{k}\parallel\delta^{\left(i\right)}\left(f\left(Y\right)\square X-\tilde{f\left(Y\right)}\square\tilde{X}\right)\parallel_{2\beta}q_{\beta}\left(\mathbf{x},\mathbf{y}\right)p_{\beta}\left(\mathbf{x},\mathbf{y}\right).
\]
 using techniques as in the proof of Theorem \ref{thm:Young integration}
one can see that 
\[
\parallel\delta^{\left(i\right)}\left(f\left(Y\right)\square X-\tilde{f\left(Y\right)}\square\tilde{X}\right)\parallel_{2\beta}
\]
\[
\leq C\left(\parallel f\left(Y\right)-f\left(\tilde{Y}\right)\parallel_{\beta,+}\parallel X\parallel_{\beta,\square}+\parallel f\left(\tilde{Y}\right)\parallel_{\beta,+}\parallel X-\tilde{X}\parallel_{\beta,\square}\right),
\]
Furthermore, we can see that 
\[
f\left(Y\left(\mathbf{x}\right)\right)\square_{\mathbf{x},\mathbf{y}}X-f\left(\tilde{Y}\left(\mathbf{x}\right)\right)\square_{\mathbf{x},\mathbf{y}}\tilde{X}
\]
\[
=\left(f\left(Y\left(\mathbf{x}\right)\right)-f\left(\tilde{Y}\left(\mathbf{x}\right)\right)\right)\square_{\mathbf{x},\mathbf{y}}X+f\left(\tilde{Y}\left(\mathbf{x}\right)\right)\left(\square_{\mathbf{x},\mathbf{y}}X-\square_{\mathbf{x},\mathbf{y}}\tilde{X}\right).
\]
and we know that 
\[
\parallel f\left(Y\right)-f\left(\tilde{Y}\right)\parallel_{\infty}\leq|f\left(Y\left(\mathbf{0}\right)\right)-f\left(\tilde{Y}\left(\mathbf{0}\right)\right)|+\parallel f\left(Y\right)-f\left(\tilde{Y}\right)\parallel_{\beta},
\]
 and using Equation (\ref{eq:function of young field bound}) we can
see
\[
\parallel f\left(Y\right)-f\left(\tilde{Y}\right)\parallel_{\infty}\leq C_{\beta,M}\parallel f\parallel_{C_{b}^{2}}\left(|\xi\left(\mathbf{0}\right)-\tilde{\xi}\left(\mathbf{0}\right)|+\parallel f\parallel_{C_{b}^{2}}\parallel Y-\tilde{Y}\parallel_{\beta}\right).
\]
Also, it is clear that
\[
|\square_{\mathbf{x},\mathbf{y}}X-\square_{\mathbf{x},\mathbf{y}}\tilde{X}|\leq\parallel X-\tilde{X}\parallel_{\beta,\square}q_{\beta}\left(\mathbf{x},\mathbf{y}\right).
\]
 and therefore 
\[
\parallel Y-\tilde{Y}\parallel_{\beta}
\]
\[
\leq C_{M,f}\left(|\xi\left(\mathbf{0}\right)-\tilde{\xi}\left(\mathbf{0}\right)|+\parallel\xi-\tilde{\xi}\parallel_{\beta}+\parallel X-\tilde{X}\parallel_{\beta,\square}+\parallel Y-\tilde{Y}\parallel_{\beta}p_{\alpha-\beta}\left(\mathbf{0},\mathbf{T}\right)\right).
\]
From and we again choose $\mathbf{T}=\mathbf{\frac{1}{2}}$ such that
for all hyper-cubes $\left[\rho,\rho+\mathbf{T}\right]$ 
\[
\parallel Y-\tilde{Y}\parallel_{\beta;\left[\rho,\rho+\mathbf{T}\right]}\leq C_{M,f}2\left(|\xi\left(\mathbf{0}\right)-\tilde{\xi}\left(\mathbf{0}\right)|+\parallel\xi-\tilde{\xi}\parallel_{\beta}+\parallel X-\tilde{X}\parallel_{\beta,\square}\right),
\]
then by the scalability of Hölder norms from Lemma \ref{prop:scaleability of H=0000F6lder norms},
we obtain the desired result.
\end{proof}
\bibliographystyle{plain}

\end{document}